\documentclass[10pt]{amsart}
\usepackage{verbatim}
\usepackage{eucal,url,amssymb,stmaryrd,enumerate,amscd}
\usepackage{amsmath}
\usepackage[pagebackref,colorlinks=true,linkcolor=blue,citecolor=blue]{hyperref}
\usepackage{amsfonts}
\usepackage{amsthm}
\usepackage[margin=1in]{geometry}
\usepackage{tensor}
\numberwithin{equation}{section}
\newtheorem{thrm}{Theorem}[section]
\newtheorem{lemma}[thrm]{Lemma}
\newtheorem{prop}[thrm]{Proposition}
\newtheorem{cor}[thrm]{Corollary}
\newtheorem{dfn}[thrm]{Definition}

\newtheorem{conv}[thrm]{Convention}

\newcommand{\veps}{\varepsilon}

\newcommand{\QH}{\boldsymbol {G\,(\mathbb{H})}}

\newcommand{\talpha}{\tilde{\alpha}}
\newcommand{\tbeta}{\tilde{\beta}}
\newcommand{\tgamma}{\tilde{\gamma}}

 1

\newcommand{\vol}{\, Vol_{\eta}}

\newcommand{\W}{W^{qc}}

\begin{document}

\begin{abstract} It is shown that the qc Yamabe problem has a solution on any compact qc manifold which is non-locally qc equivalent to the standard 3-Sasakian sphere.
Namely, it is proved that on a compact non-locally spherical qc manifold there exists a qc conformal qc structure with constant qc scalar curvature.
\end{abstract}

\keywords{}

\title[The qc Yamabe problem on non-spherical qc manifolds]{The qc Yamabe problem on non-spherical quaternionic contact manifolds}
\date{\today }
\author{S. Ivanov}
\address[Stefan Ivanov]{University of Sofia, Faculty of Mathematics and
Informatics, blvd. James Bourchier 5, 1164, Sofia, Bulgaria}
\address{and Institute of Mathematics and Informatics, Bulgarian Academy of
Sciences} \email{ivanovsp@fmi.uni-sofia.bg}
\author{A. Petkov}
\address[Alexander Petkov]{University of Sofia, Faculty of Mathematics and
Informatics, blvd. James Bourchier 5, 1164, Sofia, Bulgaria}
\email{a\_petkov\_fmi@abv.bg}

\maketitle
\tableofcontents

\setcounter{tocdepth}{2}

\section{Introduction}

It is well known that the sphere at infinity of a  non-compact
symmetric space $M$ of rank one carries a natural
Carnot-Carath\'eodory structure, see \cite{M,P}.  In the real
hyperbolic case one obtains the conformal class of the round
metric on the sphere. In the remaining cases, each of the complex,
quaternionic and octonionic   hyperbolic metrics on the unit ball
induces a Carnot-Carath\'eodory  structure on the unit sphere.
This defines a conformal structure  on a sub-bundle of  the
tangent bundle of co-dimension $\dim_{\mathbb{R}}\mathbb{K}-1$,
where $\mathbb{K}=\mathbb{C},\,  \mathbb{H},\,  \mathbb{O}$. In
the complex case the obtained geometry is the well studied
standard CR structure on the unit sphere in complex space. In the quaternionic case one arrives at the notion of a quaternionic contact structure.
The quaternionic contact (qc) structures were introduced by O. Biquard,
see \cite{Biq1}, and are modeled on  the conformal boundary at
infinity of the quaternionic hyperbolic space. Biquard showed that
the infinite dimensional family \cite{LeB91} of complete
quaternionic-K\"ahler deformations of the quaternionic hyperbolic
metric have conformal infinities which provide an infinite
dimensional family of examples of qc structures. Conversely,
according to \cite{Biq1} every real analytic qc structure on a
manifold $M$ of dimension at least eleven is the conformal
infinity of a unique quaternionic-K\"ahler metric defined in a
neighborhood of $M$. Furthermore,  \cite{Biq1} considered CR and
qc structures as boundaries at infinity of Einstein metrics rather
than only as boundaries at infinity of  K\"ahler-Einstein and
quaternionic-K\"ahler metrics, respectively.  In fact, in
\cite{Biq1} it was shown that in each of the three cases (complex,
quaternionic, octoninoic)  any small perturbation of the standard
Carnot-Carath\'eodory structure on the boundary is the conformal
infinity of an essentially unique Einstein metric on the unit
ball, which is asymptotically symmetric.  In the Riemannian case
the corresponding question was posed  in \cite{FGr85} and  the
perturbation result was proven in \cite{GrL91}.

There is a deep analogy between the geometry of  qc manifolds and
the geometry of strictly pseudoconvex CR manifolds as weel as the
geometry of conformal Riemannian manifolds. The qc structures, appearing  as the boundaries at infinity of
asymptotically hyperbolic quaternionic manifolds,  generalize to
the quaternion algebra the sequence of families of geometric
structures that are the boundaries at infinity of real and complex
asymptotically hyperbolic spaces. In the real case, these
manifolds are simply conformal manifolds and in the complex case, the boundary
structure is that of a CR manifold.

A natural extension of the Riemannian and the CR Yamabe problems is
the quaternionic contact  (qc) Yamabe problem, a particular case of
which \cite{GV,Wei,IMV,IMV1}  amounts to finding  the best
constant in the $L^2$ Folland-Stein Sobolev-type embedding and the
functions for which the equality is achieved, \cite{F2} and
\cite{FS}, with a complete solution  on the  quaternionic
Heisenberg groups  given in \cite{IMV1,IMV2,IMV4}.

{Following Biquard, a quaternionic contact structure
(\emph{qc structure}) on a real (4n+3)-dimensional manifold $M$ is
a codimension three distribution $H$ (\emph{the horizontal distribution}) locally given as the kernel of a $%
\mathbb{R}^3$-valued one-form $\eta=(\eta_1,\eta_2,\eta_3)$, such
that, the three two-forms $d\eta_i|_H$ are the fundamental forms
of a quaternionic Hermitian structure on $H$. The 1-form $\eta$
is determined up to a conformal factor and the action of $SO(3)$ on $\mathbb{R}^3$, and therefore $%
H$ is equipped with a conformal class $[g]$ of quaternionic
Hermitian metrics.

For a qc manifold of crucial importance is the existence of a
distinguished linear connection  $\nabla$ preserving the qc
structure, defined by O. Biquard in \cite{Biq1}, and its scalar
curvature $S$,  called  qc-scalar curvature. The Biquard
connection  plays a role similar to the Tanaka-Webster connection
\cite{W} and \cite{T} in the CR case. A natural question coming
from the conformal freedom of the qc structures is the
\emph{quaternionic contact Yamabe problem \cite{Wei,IMV,IV2}: The
qc Yamabe problem on a compact qc manifold $M$ is the problem  of
finding a metric $\bar g\in [g]$ on $H$ for which the qc-scalar
curvature is constant.}

The question reduces to the solvability of the quaternionic
contact (qc) Yamabe equation
\begin{equation*}
\mathcal{L} f:=4\frac {Q+2}{Q-2}\ \triangle f -\ f\, S \
=\ - \ f^{2^*-1}\ \overline{S},
\end{equation*}
where $\triangle $ is the horizontal sub-Laplacian, $\triangle f\
=\ tr^g(\nabla^2f)$,  $S$ and $\overline{S}$ are the qc-scalar
curvatures
correspondingly of $(M,\, \eta)$ and $(M, \, \bar\eta=f^{4/(Q-2)}\eta)$, where $2^* = \frac {2Q%
}{Q-2},$ with $Q=4n+6$-the homogeneous dimension.

Complete solutions to the qc Yamabe equation on the 3-Sasakian
sphere, and more general, on any compact 3-Sasakian manifold were
found in \cite{IMV4}. The case of the sphere is rather important
for the general solution of the qc Yamabe problem since it
provides a family of  "test functions" used in attacking the
general case.

The qc  Yamabe problem  is of variational nature as we remind next (see e.g. \cite{Wei,IMV,IV2,IV5}).
Given a quaternionic contact (qc) manifold $(M,[\eta])$ with a fixed conformal class defined by a quaternionic contact form $\eta$,
solutions to the quaternionic contact Yamabe problem are critical points of the \emph{qc Yamabe functional}
\begin{equation*}\label{Yamabefunc}
\Upsilon_M(\eta):=\frac{\int_MS\,Vol_{\eta}}{\big(\int_M\,Vol_{\eta}\big)^{2/2^*}},
\end{equation*}
where  $Vol_{\eta}$ is the natural volume form, associated to the contact form $\eta,$ and $2^*:=\frac{2n+3}{n+1}.$ The \emph{qc Yamabe constant} is defined by
\begin{equation*}\label{Yamabeconst}
\lambda(M):=\lambda(M,[\eta])=\underset{\eta}{\mathrm{inf}}\Upsilon_M(\eta).
\end{equation*}
If $\eta$ is a fixed qc contact form one considers  the functional (which is also called \emph{qc Yamabe functional})
\begin{equation}\label{Yamabefunc1}
\Upsilon_{\eta}(f):=\Upsilon_M(f^{2^*-2}\eta) =\frac{\ \int_M\Bigl(4\frac {Q+2}{Q-2}\ \lvert \nabla f
\rvert^2_\eta\ +\ {S}\, f^2\Bigr)\,Vol_{\eta}}{\big(\int_Mf^{2^*}\,Vol_{\eta}\big)^{2/2^*}}
=\frac{\int_M\big(4\frac{n+2}{n+1}|\nabla f|_\eta^2+Sf^2\big)\,Vol_{\eta}}{\big(\int_Mf^{2^*}\,Vol_{\eta}\big)^{2/2^*}},
\end{equation}
where $0<f\in C^{\infty}(M)$, $\nabla f$ is the horizontal gradient of $f$. The qc Yamabe constant can be expressed as
\begin{equation*}
\lambda(M)=\underset{f}{\mathrm{inf}}\Upsilon_{\eta}(f)
\end{equation*}
and the qc Yamabe equation characterizes the
non-negative extremals of the qc Yamabe functional \eqref{Yamabefunc1}.

The main result of W. Wang  \cite{Wei} states that the qc Yamabe constant of a compact quaternionic contact manifold is always less or equal than that of the standard 3-Sasakian sphere,
$\lambda(M)\leq\lambda(S^{4n+3})$ and, if the constant is strictly less than that of the sphere, the qc Yamabe problem has a solution,
i.e. there exists a global qc conformal transformation sending the given qc structure to a qc structure with constant qc scalar curvature.

The qc Yamabe constant on the standard unit 3-Sasakian sphere is calculated in \cite{IMV1,IMV2,IMV4}. It is shown \cite[Theorem~1.1]{IMV2} that
$$\Lambda=\lambda(S^{4n+3})=\frac{16n(n+2)2^{\frac1{2n+3}}}{(2n+1)^{\frac1{2n+3}}}\pi^{\frac{2n+2}{2n+3}}.
$$

Guided by the conformal and CR cases, a natural conjecture is that
the qc Yamabe constant of every compact locally
non-flat qc manifold (in conformal quaternionic contact sense) is
strictly less than the qc Yamabe constant of the
sphere with its standard 3-Sasakian qc structure
\cite{IMV,IV2}.

The purpose of this paper is to confirm the above conjecture. Our
main result  is
\begin{thrm}\label{mainth}
Suppose $M$ is a compact qc manifold of dimension $4n+3$. If $M$
is not locally qc equivalent to the standard 3-Sasakian structure
on the sphere $S^{4n+3}$ then $\lambda(M)<\Lambda$, and thus the qc Yamabe problem can be solved
on $M$.
\end{thrm}
This is analogous to the result of T. Aubin \cite{Au} for the
Riemannian version of the Yamabe problem: Every compact Riemannian
manifold of  dimension bigger than 5 which is not locally
conformally flat posses a conformal metric of constant scalar
curvature and the result of D. Jerison \& J. M. Lee \cite{JL} for
the CR version of the Yamabe problem: Every compact CR manifold of
dimension bigger than $3$ which is not locally CR equivalent to the
sphere posses a conformal pseudohermitian metric of constant
pseudohermitian scalar curvature.  Aubin's result is limited to
dimensions bigger than $5$. In the remaining cases the problem was
solved by R. Schoen in \cite{Sh} (see also \cite{LP} as well as
\cite{Bah,BahBre} for a different approach to the Riemannian
Yamabe problem). Similarly, Jerison and Lee's result is limited to
dimensions bigger than $3$ and in the remaining cases the problem
was solved by N. Gamara in \cite{Ga} and N. Gamara and R. Yacoub
in \cite{GaY}.

\emph{Surprisingly in our theorem for the qc case there is no
dimensional restriction which is a bit  different with the
Riemannian and the CR cases.}

To achieve  the result we follow and adapt to the qc case the
steps of the Jerison \& Lee's theorem \cite{JL} which solves the
CR Yamabe problem on non-spherical CR manifold of dimension bigger
than $3$. The main idea (as in \cite{JL}) is to find a precise
asymptotic expression of the qc Yamabe functional.
Our efforts throughout the present paper are concentrated on the establishment of a local asymptotic expression of the qc Yamabe functional \eqref{Yamabefunc1}
in terms of the
Yamabe invariant of the sphere and the norm of the qc conformal
curvature $W^{qc}$ introduced by Ivanov-Vassilev in \cite{IV}. The
next step is to show that the term in front of $||W^{qc}||^2$ in
this expression is a negative constant and then apply the  qc
conformal flatness theorem  which states that $W^{qc}=0$ if and only if the qc
manifold is locally qc equivalent to the standard 3-Sasakian
sphere \cite[Theorem~1.2, Corollary~1.3]{IV}.

The key idea is that the 3-Sasakian sphere posses a one-parameter
family of extremal qc contact forms that concentrate near a point.
Instead of the sphere one uses as a model the quaternionic
Heisenberg group $\QH=\mathbb H^n\times Im(\mathbb H)$. The Cayley
transform \cite{IMV} gives a qc-equivalence between $\QH$ and the
sphere minus a point, which allows us to think of the standard
spherical 3-Sasakian form as a qc form on $\QH$.

The Heisenberg group carries a natural family of parabolic
dilations: for $s > 0$, the map
$\delta_s(x^{\alpha},t^i)=(sx^{\alpha},s^2t^i)$ is a qc
automorphism of $\QH$. These dilations give rise to a family of
extremal qc contact forms $\Theta^{\veps} =
\delta^*_{1/{\veps}}\Theta$ on $\QH$ which become
more and more concentrated near the origin as $\veps
\rightarrow 0$. We show that the qc Yamabe functional $\Upsilon_M$
is closely approximated by $\Upsilon_{\QH}$ for contact forms
supported very near the base point.

 We present a precise
asymptotic expression for $\Upsilon_M(\eta^{\veps})$ as
$\veps\rightarrow 0$ for a suitably  chosen qc contact forms
$\eta^{\veps}$. To this end we use the intrinsic qc normal
coordinates introduced by Ch. Kunkel in \cite{Kunk}. The main
ingredient in the Kunkel's construction of these coordinates is
the fact that the tangent space of a quaternionic contact manifold
has a natural parabolic dilation, instead of the more common
linear dilation seen in Riemannian geometry. Kunkel used these
curves to define an exponential map from the tangent space at a
point to the base manifold and showed that this exponential map
incorporate this parabolic structure. Using this map and a special
frame at the center point Kunkel was able to construct a set of
parabolic normal coordinates. Using these parabolic normal
coordinates and the effect of a conformal change of qc contact
structure on the curvature and torsion tensors of the Biquard
connection Kunkel was able to define a function that, when used as
the conformal factor, causes the symmetrized covariant derivatives
of a certain tensor constructed from the curvature and torsion to
vanish. More precisely, using invariance theory, Kunkel showed in
\cite{Kunk} that the only remaining term of weight less than or
equal to $4$ that does not necessarily vanish at the center point
is the squared norm of the qc conformal curvature, $||W^{qc}||^2$.

Using the above coordinates around any fixed point $q\in M$ of an arbitrary qc manifold $(M,[\eta])$  we define the "test forms"
$\eta^{\veps}=(f^{\veps})^{2^*-2}\eta$, where $\eta$ is normalized at $q$ qc contact form and $f^{\veps}$ is a suitable "test function"
inspired from the solution of the qc Yamabe equation on the 3-Sasakian sphere found in \cite{IMV2,IMV4} and compute an asymptotic formula for
$\Upsilon_M(\eta^{\veps})$ as $\veps \rightarrow 0$.

We
show  in Section~\ref{explicit}, Theorem~\ref{expansion}, that 
the next  formula holds
\begin{equation}\label{expansionfunc1}
\Upsilon_M(\eta^{\veps})=\Upsilon_{\eta}(f^{\veps})=\Lambda(1-c(n)||W^{qc}(q)||^2\veps^4)+O(\veps^5),
\end{equation}
where $c(n)$ is a positive dimensional constant.

Finally, we compute the exact value of the constant $c(n)$ and
show that it is strictly positive. Since, $W^{qc}$ is identically
zero precisely when M is locally qc equivalent to the sphere
\cite{IV}, under the hypotheses of Theorem~\ref{mainth} there is a
point $q\in M$ where $W^{qc}(q)\not=0$. This implies that for
$\veps$ small enough we can achieve
$\Upsilon_M(\eta^{\veps})<\Lambda$ and  applying the main result of W. Wang \cite{Wei} we prove Theorem~\ref{mainth}.

\begin{conv}
\label{conven} \hfill\break\vspace{-15pt}
\begin{enumerate}[a)]
\item We shall use the following index convention:\\$a,b,c,...\in\{1,...,4n+3\}$, $\alpha,\beta,\gamma,...\in\{1,...,4n\}$, $\tilde{\alpha},\tilde{\beta},\tilde{\gamma},...\in\{4n+1,4n+2,4n+3\}$, $i,j,k,...\in\{1,2,3\}$.
\item The summation convention over repeated indices will be used (unless otherwise stated). For example, $A_{\alpha\beta\alpha\gamma}=\sum_{\alpha=1}^{4n}A_{\alpha\beta\alpha\gamma}$ and $B_{i}C^{i}=\sum_{i=1}^3B_{i}C^{i}$ and s.o.
\end{enumerate}
\end{conv}

\textbf{Acknowledgments} We thank   I. Minchev and D. Vassilev
for many useful discussions. The research is partially supported
by Contract DFNI I02/4/12.12.2014 and  Contract 195/2016 with the
Sofia University "St.Kl.Ohridski".

\section{Quaternionic contact manifolds}

In this section we will briefly review the basic notions of
quaternionic contact geometry and recall some results from
\cite{Biq1}, \cite{IMV} and \cite{IV} which we will use in this
paper. Since we will  work with the Kunkel's  qc parabolic normal
coordinates, we will follow the notations in \cite{Kunk}.

\subsection{Quaternionic contact structures and the Biquard connection}
A quaternionic contact (qc) manifold $(M, \eta,g, \mathbb{Q})$ is a $4n+3$%
-dimensional manifold $M$ with a codimension three distribution
$H$ locally given as the kernel of a 1-form
$\eta=(\eta^1,\eta^2,\eta^3)$ with values in $\mathbb{R}^3$. In
addition $H$ has an $Sp(n)Sp(1)$ structure, that is, it is
equipped with a Riemannian metric $g$ and a rank-three bundle
$\mathbb{Q}$ consisting of endomorphisms of $H$ locally generated
by three almost complex structures $I^1,I^2,I^3$ on $H$ satisfying
the identities of the imaginary unit quaternions,
$I^1I^2=-I^2I^1=I^3, \quad I^1I^2I^3=-id_{|_H}$ which are
hermitian compatible with the metric $g(I^i.,I^i.)=g(.,.)$ with the
 compatibility condition  $\quad 2g(I^iX,Y) =
d\eta^i(X,Y),  i=1,2,3, \quad X,Y\in H.$

The transformations preserving a given qc
structure $\eta$, i.e., $\bar\eta=f\Psi\eta$ for a positive
smooth function $f$ and an $SO(3)$ matrix $\Psi$ with smooth
functions as entries are called \emph{quaternionic contact
conformal (qc-conformal) transformations}. If the function $f$
is constant $\bar\eta$ is called qc-homothetic to $\eta$. The qc
conformal curvature tensor $W^{qc}$, introduced in \cite{IV}, is
the obstruction for a qc structure to be locally qc conformal to
the standard 3-Sasakian structure on the $(4n+3)$-dimensional
sphere \cite{IMV,IV}.

A special phenomena, noted in \cite{Biq1}, is that the contact
form $\eta$ determines the quaternionic structure and the metric
on the horizontal distribution in a unique way.

On a qc manifold with a fixed metric $g$ on $H$ there exists a
canonical
connection defined first by Biquard in \cite{Biq1} when the dimension $(4n+3)>7$, and in \cite%
{D} for the 7-dimensional case. Biquard showed that there  is a
unique connection $\nabla$ with torsion $T$  and a unique
supplementary subspace $V$ to $H$ in $TM$, such that:
\begin{enumerate}[(i)]
\item $\nabla$ preserves the decomposition $H\oplus V$ and the $
Sp(n)Sp(1)$ structure on $H$, i.e. $\nabla g=0, \nabla\sigma
\in\Gamma( \mathbb{Q})$ for a section
$\sigma\in\Gamma(\mathbb{Q})$, and its torsion on $H$ is given by
$T(X,Y)=-[X,Y]_{|V}, \quad X,Y\in H$; \item for $R\in V$, the endomorphism
$T(R,.)_{|H}$ of $H$ lies in $ (sp(n)\oplus sp(1))^{\bot}\subset
gl(4n)$; \item the connection on $V$ is induced by the natural
identification $ \varphi$ of $V$ with the subspace $sp(1)$ of the
endomorphisms of $H$, i.e. $ \nabla\varphi=0$.
\end{enumerate}
This canonical connection is also known as the \emph{Biquard
connection}. When the dimension of $M$ is
at least eleven Biquard \cite{Biq1} also described the supplementary distribution $V$%
, which is (locally) generated by the so called Reeb vector fields $%
\{R_1,R_2,R_3\}$ determined by
\begin{equation}  \label{bi1}
\eta^i(R_j)=\delta^i_j, \qquad (R_i\lrcorner d\eta^i)_{|H}=0\quad
\textnormal{with no summation}, \qquad (R_i\lrcorner
d\eta^j)_{|H}=-(R_j\lrcorner d\eta^i)_{|H},
\end{equation}
where $\lrcorner$ denotes the interior multiplication. If the dimension of $%
M $ is seven Duchemin shows in \cite{D} that if we assume, in
addition, the existence of Reeb vector fields as in \eqref{bi1},
then   the Biquard result  holds. \emph{Henceforth, by a qc
structure in dimension $7$ we shall mean a qc structure satisfying
\eqref{bi1}}.

Letting $$\omega\indices{_i}{^j}=R_i\lrcorner d\eta^j|_H\quad\textnormal{we have}\quad \omega\indices{_i}{^j}=-\omega\indices{_j}{^i}$$ and these are the $sp(1)$-connection 1-forms of $\nabla$,
i.e. the restriction to $H$ of the connection 1-forms for $\nabla$
on $V$. The isomorphism $\varphi$ is then simply
$\varphi(R_i)=I_i$ and these are the connection 1-forms on
$\mathbb Q$.

Notice that equations \eqref{bi1} are invariant under the natural
$SO(3)$
action. Using the triple of Reeb vector fields we extend the metric $g$ on $%
H $ to a Riemannian metric on $TM$ by requiring
$span\{R_1,R_2,R_3\}=V\perp H.$ The extended Riemannian metric
$g\oplus\sum(\eta^i)^2$ as well as the Biquard connection do not
depend on the action of $SO(3)$ on $V$, but both change  if $\eta$
is multiplied by a conformal factor \cite{IMV}. Clearly, the
Biquard connection preserves the Riemannian metric on $TM$.

Consider a frame $\{\xi_{\alpha}, R_i\}_{\alpha=1,...,4n;i=1,2,3},$
where $\{\xi_{\alpha}\}$ is an Sp(n)Sp(1) frame for H, and
$R_i$ are the three Reeb vector fields described above. It is
occasionally convenient to have a notation for the entire frame;
therefore as necessary we may refer to $R_i$ as
$\xi_{4n+i}$. In order to have a consistent
index notation we will use different letters for different ranges
of indices as in Convention~\ref{conven}.  For the dual basis we use
$\theta^{\alpha}$ and $\eta^i$. We note that both $H$ and
$V$ are orientable. The horizontal bundle is orientable since it
admits an $Sp(n)Sp(1)\subset SO(4n)$ structure and $\mathbb Q$ has
an $SO(3)$ structure, hence so does $V$. The natural volume form
on $V$ is given by $\epsilon=\eta^1\wedge\eta^2\wedge\eta^3$ ,
and this tensor provides a handy isomorphism between $V$ and
$\Lambda^2V$ (or their duals). We also denote the volume form on
$H$ by $\Omega$ and the volume form on $TM$ as $dv =\Omega\wedge
\epsilon$. Using the volume form $\epsilon^{ijk}$ and the
metrics on $H$ and $V$, there is a convenient way to express
composition of the almost complex structures and contractions of
the volume form with itself,
\begin{gather*}\label{e1}
I\indices{_i}{^{\alpha}}{_{\gamma}}I\indices{_j}{^{\gamma}}{_{\beta}}=-g_{ij}\delta^{\alpha}_{\beta}+\epsilon_{ijk}I\indices{^{k\alpha}}{_{\beta}};\\\label{e2}
\epsilon_{ijk}\epsilon^{ilm}=\delta^l_j\delta^m_k-\delta^l_k\delta^m_j;
\quad \epsilon_{ijk}\epsilon^{ijl}=2\delta^l_k.
\end{gather*}

\subsection{Torsion and curvature} Since $T(X,Y)=-[X,Y]_{|V}\in V \quad\textnormal{for}\quad X,Y \in H$, we have
\[ T\indices{^i}{_{\alpha\beta}}=-\eta^i([\xi_{\alpha},\xi_{\beta}])=d\eta^i(\xi_{\alpha},\xi_{\beta})=2g(I^i\xi_{\alpha},\xi_{\beta})=-2I\indices{^i}{_{\alpha\beta}},\quad\textnormal{and}\quad T\indices{^{\alpha}}{_\beta\gamma}=0.\]

The properties of the Biquard connection are encoded in the
properties of the torsion endomorphism $T_{R }=T(R ,\cdot
):H\rightarrow H,\quad R \in V$. Decomposing the endomorphism
$T_{R }\in (sp(n)+sp(1))^{\perp }$ into its symmetric part
$T_{R}^{0}$ and skew-symmetric part $b_{R },T_{R }=T_{R}^{0}+b_{R
}$, O. Biquard shows in \cite{Biq1} that the torsion $T_{R}$ is
completely trace-free, $tr\,T_{R }=tr\,(T_{R}\circ I^i)=0$, its
symmetric part has the properties $T_{R
_{i}}^{0}I^{i}=-I^{i}T_{R_{i}}^{0}\quad I^{2}(T_{R
_{2}}^{0})^{+--}=I^{1}(T_{R _{1}}^{0})^{-+-},\quad I^{3}(T_{R
_{3}}^{0})^{-+-}=I^{2}(T_{R_{2}}^{0})^{--+},\quad I^{1}(T_{R
_{1}}^{0})^{--+}=I^{3}(T_{R _{3}}^{0})^{+--}$, where the
superscript $+++$
means commuting with all three $I^{i}$, $+--$ indicates commuting with $%
I^{1} $ and anti-commuting with the other two and etc. The
skew-symmetric part can be represented as $b_{R _{i}}=I^{i}\mu$,
where $\mu$ is a traceless symmetric (1,1)-tensor on $H$ which
commutes with $I^{1},I^{2},I^{3}$. Therefore we have
$T_{R_{i}}=T_{R_{i}}^{0}+I^{i}\mu$.  The symmetric, trace-free endomorphism on $H$ defined by
$\tau= (T_{R_1}^{0}I_1+T_{R_2}^{0}I_2+T_{ R_3}^{0}I_3)$ \cite{IMV}
determines completely the symmetric part
\cite[Proposition~2.3]{IV} $4T^0_{R_i}=I^i\tau-\tau I^i$. Thus
$\tau$  together with $\mu$ are the two $Sp(n)Sp(1)$-invariant
components of the torsion endomorphism. If $n=1$ then the tensor $\mu$
vanishes identically, $\mu=0$, and the torsion is a symmetric
tensor, $T_{R}=T_{R}^{0}$. Consider the Casimir
operator $\mathcal C=\sum_{i=1}^3I^i\otimes I^i$ one has $\mathcal{
C}^2=2\mathcal C +3$ and $\tau$ belongs to the eigenspace of
$\mathcal C$ corresponding to the eigenvalue $-1$ while $\mu$
belongs to the eigenspace of $\mathcal C$ corresponding to the
eigenvalue $3$, $\mathcal C\tau =-\tau, \quad \mathcal C\mu=3\mu$.

Further in the  paper we use the index convention. For example,
$\mathcal
C\indices{^{\alpha\gamma}}{_{\beta\delta}}=I\indices{_i}{^\alpha}{_\beta}I\indices{^{i\gamma}}{_\delta}$ and  \cite{Biq1}
\begin{equation}\label{torend}
T\indices{^i}{_{\alpha\beta}}=-2I\indices{^i}{_{\alpha\beta}},
\quad
T\indices{^\alpha}{_{i\alpha}}=0=T\indices{^\alpha}{_{i\beta}}I\indices{^\beta}{_\alpha},
\quad\textnormal{for any}\quad I\in \mathbb Q, \quad
I\indices{_i}{^\alpha}{_\gamma}\mu\indices{^\gamma}{_\beta}=\mu\indices{^\alpha}{_\gamma}I\indices{_i}{^\gamma}{_\beta}.
\end{equation}

We denote $R\indices{_{abc}}{^d}=\theta^d(R(\xi_a,\xi_b)\xi_c)$ for
the curvature tensor of the Biguard connection. The horizontal
Ricci tensor, called \emph{qc Ricci tensor}, is
$Ric=R_{\alpha\beta}=R\indices{_{\gamma\alpha\beta}}{^\gamma}$ and
the qc scalar curvature is $S=R\indices{_\alpha}{^\alpha}.$ There are nine Ricci type tensors obtained by certain contractions of the
curvature tensor against the almost complex structures  introduced in \cite[Definition~3.7, Definition~3.9]{IMV} by
\begin{equation}\label{CurvatureContractAlmost}
\rho_{iab}=\frac{1}{4n}R_{ab\alpha\beta}I\indices{_i^{\beta\alpha}},\quad \zeta_{iab}=\frac{1}{4n}R_{\alpha ab\beta}I\indices{_i^{\beta\alpha}},\quad \sigma_{iab}=\frac{1}{4n}R_{\alpha\beta ab}I\indices{_i^{\beta\alpha}}.
\end{equation}
The curvature tensor $R_{ab\alpha\beta}$ decomposes as
$R_{ab\alpha\beta}=\mathcal R_{ab\alpha\beta}+\rho_{iab}I\indices{^i}{_{\alpha\beta}}$, where $\mathcal R_{ab\alpha\beta}$
is the $\frak{sp}(n)$-component of $R_{ab\alpha\beta}$ and commutes with the almost complex structures in the second pair of indices \cite[Lemma~3.8]{IMV}.

In fact, according to \cite[Theorem~3.11]{IV} the whole curvature $R_{abcd}$ is completely determined by  its horizontal part $R_{\alpha\beta\gamma\delta}$,
the symmetric horizontal tensors $\tau_{\alpha\beta},  \mu_{\alpha\beta}$, the qc scalar curvature $S$ and their covariant derivatives up to the second order.

We collect below the necessary facts from \cite[Lemma~3.11, Theorem~3.12,
Corollary~3.14, Theorem~4.8]{IMV} and \cite[Proposition~2.3, Theorem~2.4]{IV}.
\begin{equation}\label{tor0}
\begin{aligned}
&T\indices{^\alpha}{_{i\beta}}=(T^0)\indices{_i}{^\alpha}{_\beta}+I\indices{_i}{^\alpha}{_\gamma}\mu\indices{^\gamma}{_\beta}=
\frac14\Big(\tau\indices{^\alpha}{_\gamma}I\indices{_i}{^\gamma}{_\beta}+I\indices{_i}{^\alpha}{_\gamma}\tau\indices{^\gamma}{_\beta}\Big)
+I\indices{_i}{^\alpha}{_\gamma}\mu\indices{^\gamma}{_\beta};\\
&T\indices{^\alpha}{_{ij}}=d\theta^{\alpha}(R_i,R_j); \quad
T\indices{^k}{_{ij}}=-\frac{S}{8n(n+2)}\epsilon\indices{_{ij}^k};
\quad T\indices{^\alpha}{_{\beta\gamma}}=0=
T\indices{^i}{_{j\alpha}};\\
&R_{\alpha\beta}=(2n+2)\tau_{\alpha\beta}+2(2n+5)\mu_{\alpha\beta}+\frac{S}{4n}g_{\alpha\beta};\\
&\rho_{i\alpha\beta}=\frac12\Big(\tau_{\alpha\gamma}I\indices{_i}{^\gamma}{_\beta} -\tau_{\gamma\beta}I\indices{_i}{^\gamma}{_\alpha}\Big)+2\mu_{\alpha\gamma}I\indices{_i}{^\gamma}{_\beta}-\frac{S}{8n(n+2)}I_{i\alpha\beta};\\
&\zeta_{i\alpha\beta}=-\frac{2n+1}{4n}\tau_{\alpha\gamma}I\indices{_i}{^\gamma}{_\beta} +\frac1{4n}\tau_{\gamma\beta}I\indices{_i}{^\gamma}{_\alpha}+\frac{2n+1}{2n}\mu_{\alpha\gamma}I\indices{_i}{^\gamma}{_\beta}+\frac{S}{16n(n+2)}I_{i\alpha\beta};\\
&\sigma_{i\alpha\beta}=\frac{n+2}{2n}\Big(\tau_{\alpha\gamma}I\indices{_i}{^\gamma}{_\beta} -\tau_{\gamma\beta}I\indices{_i}{^\gamma}{_\alpha}\Big)-\frac{S}{8n(n+2)}I_{i\alpha\beta};\\
&0=\tau\indices{_{\alpha\beta,}}{^\beta}-6\mu\indices{_{\alpha\beta,}}{^\beta}-\frac{4n-1}2\epsilon^{ijk}T\indices{^\beta}{_{jk}}I_{i\beta\alpha}
-\frac3{16n(n+2)}S\indices{_,}{_\alpha};\\
&0=\tau\indices{_{\alpha\beta,}}{^\beta}-\frac{n+2}2\epsilon^{ijk}T\indices{^\beta}{_{jk}}I_{i\beta\alpha}
-\frac3{16n(n+2)}S\indices{_,}{_\alpha};\\
&0=\tau\indices{_{\alpha\beta,}}{^\beta}-3\mu\indices{_{\alpha\beta,}}{^\beta}+2\epsilon^{ijk}T\indices{^\beta}{_{jk}}I_{i\beta\alpha}
-R\indices{_{\gamma
i\beta}}{^\gamma}I\indices{^{i\beta}}{_\alpha}.
\end{aligned}
\end{equation}

\subsection{Conformal change of the qc structure} In this section
we recall  the conformal change of a qc structure and list the
necessary facts from \cite{IMV}. Let $u$ be a smooth
function on a $(4n+3)$--dimensional qc manifold $(M,\eta,g,\mathbb Q)$.
Let $\tilde{\eta}^i=e^{2u}\eta^i, \quad \tilde{g}=e^{2u}g$ be the
conformal transformation of the qc structure $(\eta,g)$. Then
$d\tilde{\eta}^i=2e^{2u}du\wedge\eta^i+e^{2u}d\eta^i$
which restricted to $H$ gives
$$d\tilde{\eta}^i(X,Y)=e^{2u}d\eta^i(X,Y)=2e^{2u}g(I^iX,Y)=2\tilde{g}(I^iX,Y),\quad X,Y \in H.$$
The new Reeb vector fields  given by $\tilde{R}_i=e^{-2u}\Big(R_i-I\indices{_i}{^\alpha}{_\beta}u^{\beta}\xi_{\alpha}
\Big)$ determine the new globally defined supplementary space $\tilde{V}$. Note that nevertheless the one forms $\eta_i$ are local the qc conformal deformation has a global nature since the distributions $H,V$, the bundle $\mathbb{Q}$ and the horizontal metric $g$ are globally defined. We set $\tilde{\xi}_{\alpha}=\xi_{\alpha}$ and
$\tilde{\theta}^{\alpha}=\theta^{\alpha}+I\indices{_i}{^\alpha}{_\beta}u^{\beta}\eta^i$.
Then $\tilde{\theta}^{\alpha}(\tilde{R}_i)=0$ and
$\tilde{\eta}^i(\tilde{\xi}_{\alpha})=0$. Denote by $\mathbb P_{-1}$
and $\mathbb P_{3}$ the projection onto the $(-1)$-- and the
$3$--eigenspaces of the operator $\mathcal C$. Then the torsion and
the qc scalar curvature changes, \cite[(5.5),(5.6),(5.8)]{IMV},
with $h=\frac12e^{-2u}$, as follows
\begin{gather}\label{tau}
\tilde{\tau}_{\alpha\beta}=\tau_{\alpha\beta}+\mathbb
P_{-1}(4u_{\alpha}u_{\beta}-2u_{\alpha\beta});\quad 
\tilde{\mu}_{\alpha\beta}=\mu_{\alpha\beta}+\mathbb
P_{3}(-2u_{\alpha}u_{\beta}-u_{\alpha\beta});\\\label{scal}
\tilde{S}\tilde{g}_{\alpha\beta}=Sg_{\alpha\beta}-16(n+1)(n+2)u_{\gamma}u^{\gamma}g_{\alpha\beta}-8(n+2)u\indices{_\gamma}{^\gamma}g_{\alpha\beta}.
\end{gather}

\subsubsection{The qc conformal curvature tensor} In conformal
geometry, the obstruction to conformal flatness is the
well-studied Weyl tensor, the portion of the curvature tensor that
is invariant under a conformal change of the metric and its vanishing
determines if a conformal manifold is locally conformally
equivalent to the standard sphere. Likewise, in the CR case, the
tensor which determines local CR equivalence to the CR sphere is
the Chern-Moser tensor, also determined by the curvature of the
Tanaka-Webster connection and the Chern-Moser theorem \cite{ChM}
states that its vanishing is equivalent to the local CR
equivalence to the CR sphere. And just as in the conformal case,
it is the key to finding the appropriate bound for the CR Yamabe
invariant on a CR manifold. Something similar appears in the qc
case, dubbed the quaternionic contact conformal curvature by
 Ivanov and  Vassilev in \cite{IV}. In that paper
they define a tensor $W^{qc}$ and prove that it is the conformally
invariant portion of the Biquard curvature tensor. Moreover, if
the tensor $W^{qc}$ vanishes, they prove that the qc manifold is
locally qc equivalent to the quaternionic Heisenberg group and
since the quaternionic Heisenberg group and the standard
3-Sasakian sphere are locally qc equivalent
\cite[Theorem~1.1,Theorem~1.2,Corollary~1.3]{IV}, this tensor
clearly plays the role of the Weyl or Chern-Moser tensors.

The qc conformal curvature is determined by the horizontal
curvature and the torsion of the Biguard connection as follows
\cite[(4.8)]{IV}
\begin{multline}\label{qc-tensor}
W^{qc}_{\alpha\beta\gamma\delta}=R_{\alpha\beta\gamma\delta}+g_{\alpha\gamma}L_{\beta\delta}-g_{\alpha\delta}L_{\beta\gamma}+g_{\beta\delta}L_{\alpha\gamma}-
g_{\beta\gamma}L_{\alpha\delta}\\+I_{i\alpha\gamma}L_{\beta\rho}I\indices{^{i\rho}}{_\delta}-I_{i\alpha\delta}L_{\beta\rho}I\indices{^{i\rho}}{_\gamma}+
I_{i\beta\delta}L_{\alpha\rho}I\indices{^{i\rho}}{_\gamma}-I_{i\beta\gamma}L_{\alpha\rho}I\indices{^{i\rho}}{_\delta}\\+
\frac12\Big(I_{i\alpha\beta}L_{\gamma\rho}I\indices{^{i\rho}}{_\delta}-I_{i\alpha\beta}L_{\rho\delta}I\indices{^{i\rho}}{_\gamma}+
I_{i\alpha\beta}L_{\rho\sigma}I\indices{_j}{^\rho}{_\gamma}I\indices{_k}{^\sigma}{_\delta}\epsilon^{ijk}
\Big)\\+I_{i\gamma\delta}L_{\alpha\rho}I\indices{^{i\rho}}{_\beta}-I_{i\gamma\delta}L_{\rho\beta}I\indices{^{i\rho}}{_\alpha}+
\frac1{2n}L\indices{_\rho}{^\rho}I_{i\alpha\beta}I\indices{^i}{_{\gamma\delta}},
\end{multline}
where the tensor $L_{\alpha\beta}$ is given by \cite[(4.6)]{IV}
\begin{equation}\label{tenL}
L_{\alpha\beta}=\frac12\tau_{\alpha\beta}+\mu_{\alpha\beta}+\frac{S}{32n(n+2)}g_{\alpha\beta}.
\end{equation}
Clearly the qc conformal curvature equals the horizontal
curvature  precisely when the tensor $L$ vanishes.

\subsection{The flat model--the quaternionic Heisenberg group}\label{Heis}

{The basic  example of a qc manifold is provided by the
quaternionic Heisenberg group $\QH$}  on which we introduce
coordinates by regarding
$\QH=\mathbb{H}^n\times Im(\mathbb{H})$, $\quad (q,\omega)\in
\QH$ so that the multiplication takes the form
$(q_0,w_0)\circ(q,w)=(q_0+q,w_0+w+2Im(q_0.\bar{q}))$.
Using real coordinates $(x^{\alpha}, t^i)$, the "standard" left invariant  qc contact form

\begin{equation}\label{hforms}
\Theta^i=\frac12dt^i-I\indices{^i}{_\alpha}{_\beta}x^{\alpha}dx^{\beta}.
\end{equation}
The dual left-invariant vertical Reeb  fields $T_1,T_2,T_3$ are
$
T_1=2\frac{\partial}{\partial t^1}, \quad T_2=2\frac{\partial}{\partial t^2},\quad T_3=2\frac{\partial}{\partial t^3}.
$

The left-invariant horizontal 1-forms and their dual vector fields are given by
\begin{equation}\label{invfv}
\Xi^{\alpha}=dx^{\alpha}, \qquad X_{\alpha}=\frac{\partial}{\partial x^{\alpha}}+2I\indices{^i}{_\beta}{_\alpha}x^{\beta}\frac{\partial}{\partial t^i}.
\end{equation}

The horizontal and vertical subbundles of $T\QH$ are given by $\frak h=Span\{X_{\alpha}\},\quad \frak v=Span\{T_1,T_2,T_3\}$ and $T\QH=\frak h\oplus \frak v$.
On $\QH$, the left-invariant flat connection  is the Biquard
connection, hence $\QH$ is a flat qc structure. It should be noted
that the latter property characterizes (locally) the qc structure
$\Theta$ by \cite[Proposition 4.11]{IMV}, but in fact
vanishing of the curvature on the horizontal space is enogh because of
\cite[Propositon 3.2]{IV}.

\section{QC parabolic normal coordinates}\label{qcnormalcoord}

In this Section we present the necessary facts from Kunkel's work \cite{Kunk} concerning construction of qc parabolic normal coordinates and its consequences.

The quaternionic Heisenberg group $\QH$ has a family of parabolic
dilations $(x,t)\rightarrow (sx,s^2t)$ which are automorphisms for
the Lie group $\QH$ and also for its Lie algebra. Then its tangent
space $T\QH=\frak h\oplus \frak v$  come equipped with a natural
parabolic dilation which sends a vector $(v,a)$ to $(sv,s^2a)$,
for any scalar $s$. Consider these vectors to be based at the
origin $o\in\QH$, then by moving from $o$ in the direction of the
vector $(v,a)$ for time $s$ we arrive to the parametrization of a
parabola, $s\rightarrow (sv,s^2a)$. The  parabola has a simple
expression in terms of differential equations as
$\dddot{\gamma}_{(v,a)}=0, \quad \gamma_{(v,a)}(0)=0, \quad
\dot{\gamma}_{(v,a)}(0)=v, \quad \ddot{\gamma}_{(v,a)}(0)=a.$

Extending this notion to a qc manifold with the Biquard connection produces curves that
can rightly be called parabolic geodesics, i.e. that satisfy a natural parabolic scaling
$\gamma_{(sv,s^2a)}(t) = \gamma_{(v,a)}(st).$
By appropriately restricting the initial conditions, Kunkel  showed in \cite{Kunk}  that there is a
parabolic version of the geodesic exponential map called the parabolic exponential
map. Here we state Kunkel's result for a qc manifold.
\begin{thrm}\label{parexp}[\cite{Kunk} Theorem~3.1.]
Let $(M,g,\mathbb Q)$ be a qc manifold with the Biquard connection $\nabla$ and the decomposition $TM=H\oplus V$.
Choose any $q\in M$ and $ (X, Y )\in  H_q\oplus V_q = T_qM$ be any tangent vector. Define $\gamma_{(X,Y)}$ to be the
curve beginning at $q$ satisfying
$$\nabla^2_t\dot{\gamma}_{(X,Y)}=0, \quad \gamma_{(X,Y)}(0)=q, \quad \dot{\gamma}_{(X,Y)}(0)=X, \quad \nabla_t\dot{\gamma}_{(X,Y)}(0)=Y.
$$
Then there are neighborhoods $0\in O\subset T_qM$ and $q \in O_M \subset  M$ so that the function $\Psi : O \rightarrow O_M : (X, Y ) \rightarrow \gamma_{(X,Y)}(1)$
is a diffeomorphism, and satisfies the parabolic scaling $\Psi(tX, t^2Y ) = \gamma_{(X,Y )}(t)$ wherever either side is defined.
\end{thrm}
Theorem~\ref{parexp} supplies a special frame and co-frame as
follows. Let $\{R_i\}$ be an oriented orthonormal frame for $V_q$,
and let $\{I_i\}$ be the associated almost complex structures.
Choose an orthonormal basis $\{\xi_{\alpha}\}$ for $H_q$ so that
$\xi_{4k+i+1} = I_i\xi_{4k+1}$ for $ k = 0, . . . , n-1$.
Extending these vectors to be parallel along parabolic geodesics
beginning at $q$, one obtains  a smooth local frame for $TM = H
\oplus V$. Define the dual 1-forms $\{\theta^{\alpha},\eta^i\}$ by
$\theta^{\alpha}(\xi_{\beta})=\delta^{\alpha}_{\beta},\quad\theta^{\alpha}(R_i)
= 0, \quad \eta^i(\xi_{\alpha})=0, \quad \eta^i(R_j) = \delta^i_j$
and  extending the almost complex structures by defining
$I_i\xi_{4k+1} = \xi_{4k+i+1}$ for  $k=0,\ldots,n-1,$ one gets the
Kunkel's special frame and co-frame.

Given any special frame, one defines a coordinate map on a neighborhood
of $q$ by composing the inverse of $\Psi$ with the map $\lambda : T_qM \rightarrow \mathbb{R}^{4n+3 }: X \rightarrow (x^{\alpha}, t^i) = (\theta^{\alpha}(X), \eta^i(X))$.
These coordinates are the Kunkel's  parabolic normal coordinates (called qc pseudohermitian normal coordinates in \cite{Kunk}).

The generator  of the parabolic dilations on the quaternionic Heisenberg group
$$\delta_s : (x, t)\rightarrow  (sx, s^2t) \quad \textnormal{is  the vector field} \quad P = x^{\alpha}\frac{\partial}{\partial x^{\alpha}}+2t^i\frac{\partial}{\partial t^i}.$$
A tensor field $\phi$ on $\QH$  is said to be  homogeneous of order $m$ if $\mathcal L_P\phi=m\phi$, where $\mathcal L$ is the Lie derivative.
For an arbitrary tensor field $\phi$ the symbol $\phi_{(m)}$ denotes  the part of  $\phi$ that is homogeneous of order $m$.

Given a qc manifold and parabolic normal  coordinates  centered at a
point $q\in M$ Kunkel defined the infinitesimal generator of the parabolic dilations, the vector $P$, in these coordinates and shows that it is given by [\cite{Kunk}, Lemma~3.4]
\begin{equation}\label{vecP}
P=x^{\alpha}\xi_{\alpha}+t^iR_i. \quad \textnormal{In  particular,}\quad \theta^{\alpha}(P)=x^{\alpha}, \quad \eta^i(P)=t^i,\quad \omega\indices{_a}{^b}(P)=0.
\end{equation}
Using  this result Kunkel calculated  the low order homogeneous
terms of the special co-frame and the connection 1-forms, namely
he proves
\begin{prop}\label{recurconn}[\cite{Kunk} Proposition~3.5.] In parabolic normal coordinates, the low order homogeneous terms of the special co-frame and the connection $1$-forms are
\begin{enumerate}[a)]
\vspace{0.2cm}
\item $\eta^i_{(2)}=\frac{1}{2}dt^i-I\indices{^i_{\alpha\beta}}x^{\alpha}dx^\beta;\quad \eta^i_{(3)}=0;\quad \eta^i_{(m)}=\frac{1}{m}(t^j\omega\indices{_j^i}+T\indices{^i_{jk}}t^j\eta^k-2I\indices{^i_{\alpha\beta}}x^{\alpha}\theta^{\beta})_{(m)},\quad m\geq4;$
\vspace{0.2cm}
\item  $\theta^{\alpha}_{(1)}=dx^{\alpha};\quad\theta^{\alpha}_{(2)}=0;\quad\theta^{\alpha}_{(m)}=\frac{1}{m}(x^{\beta}\omega\indices{_{\beta}^{\alpha}}-T\indices{^\alpha_{i\gamma}}x^{\gamma}\eta^i+T\indices{^{\alpha}_{i\beta}}t^i\theta^{\beta}+T\indices{^{\alpha}_{ij}}t^i\eta^j)_{(m)},\quad m\geq3;$
\vspace{0.2cm}
\item $\omega_{a(1)}^{\phantom{c}b}=0;\quad\omega_{a(m)}^{\phantom{c}b}=\frac{1}{m}(R\indices{_{\alpha\beta a}^b}x^{\alpha}\theta^{\beta}+R\indices{_{\alpha ja}^b}x^{\alpha}\eta^j-R\indices{_{\alpha ja}^b}t^j\theta^{\alpha}+R\indices{_{ija}^b}t^i\eta^j)_{(m)},\quad m\geq2.$
\end{enumerate}
\end{prop}
Following \cite{Kunk}, we denote by $\mathcal O_{(m)}$ those tensor fields whose Taylor expansions at $q$ contain
only terms of order greater than or equal to $m$. For example, from Proposition~\ref{recurconn}, $\eta^i\in\mathcal O_{(2)}$ and $\theta^{\alpha}\in \mathcal O_{(1)}.$
\begin{prop}\label{Leftinv}[\cite{Kunk} Corollary~3.6.] If we define $X_\alpha=\partial_\alpha+2I\indices{^i_{\beta\alpha}}x^{\beta}\partial_i$ and $T_i=2\partial_i$, then $\xi_\alpha=X_\alpha+\mathcal{O}_{(1)}$ and $R_i=T_i+\mathcal{O}_{(0)}.$
\end{prop}
The vector fields $\{X_1,\ldots,X_{4n},T_1,T_2,T_3\}=\{X_a\}_{a=1}^{4n+3}$ form the standard left-invariant frame on  $\mathbf{G}(\mathbb{H})$ defined in Subsection \ref{Heis} and the given frame on M is expressed as a
perturbation of them.

It is easy to check that if $\phi\in\mathcal O_{(m)}$ and $\psi\in\mathcal O_{(m')}$ then $\phi\otimes\psi\in \mathcal O_{(m+m')}$ and  we will use  the following: for any index $a$, let $ o(a) = 1$ if
$a\le 4n$ and $ o(a) = 2$ if $a > 4n$. Given a multiindex $A = (a_1, \dots,a_r)$ we let $\#A = r$ and $o(A) =\sum_{s=1}^ro(a_s)$. If we have a collection of indexed vector fields, $X_a$, we let $X_A = X_{a_r} . . .X_{a_1}$, and similarly for similar expressions.

The next facts we need from \cite{Kunk} are
\begin{prop}\label{Taylor}[\cite{Kunk} Lemma~3.8.] Let $F$ be a smooth function defined near $q\in M$. Then in parabolic normal coordinates, for any non-negative integer $m,$
\begin{equation*}
F_{(m)}=\sum_{o(A)=m}\frac{1}{(\#A)!}\Big(\frac{1}{2}\Big)^{o(A)-\#A}x^A(X_AF)|_q.
\end{equation*}
\end{prop}
\begin{lemma}\label{derK}[\cite{Kunk} Lemma~3.9]
If $\phi$ is a tensor in $\mathcal O_{(m)}$, the components of its covariant derivatives in terms of a special frame satisfy
$\quad \phi_{A,B}=X_{B}\phi_A+\mathcal O_{(m-o(AB)+2)}.$
\end{lemma}

\subsection{QC parabolic normal coordinates} Using the qc conformal properties of a qc structure Kunkel was able to determine a conformal factor helping him to normalize the parabolic normal coordinates
(qc parabolic normal coordinates) in which coordinates many tensor
invariants of the qc structure vanish at the origin [\cite{Kunk},
Section~4]. We explain briefly his results which we need in
finding asymptotic expansion of  the qc Yamabe functional.

Let $\tilde\eta=e^{2u}\eta$ for a smooth function $u$. In this subsection we denote the objects with respect to $\tilde{\eta}$ by $\tilde .$, for a object $.$ with respect to $\eta$.
Suppose that $u$ is of order $m\ge 2$ with respect to the vector $P$. Kunkel showed that the connection 1-forms  of the Biquard connection $\tilde{\omega}$ and $\omega$
corresponding to $\tilde{\eta}$ and $\eta$, respectively, satisfy
$$\tilde{\omega}\indices{_{\alpha}}{^{\beta}}=\omega\indices{_{\alpha}}{^{\beta}}+\mathcal O_{(m)}; \quad
\tilde{\omega}\indices{_i}{^j}=\omega\indices{_i}{^j}+\mathcal O_{(m)}.
$$
For $u\in\mathcal O_{(m)}$, applying \eqref{tenL}, \eqref{tau} and \eqref{scal} one gets
\begin{equation}\label{tildeL}\tilde{L}_{\alpha\beta}=L_{\alpha\beta}-u_{(\alpha\beta)}+\mathcal O_{(m-1)},
\end{equation}
where we used the common notation $u_{(\alpha\beta)}=\frac12(u_{\alpha\beta}+u_{\beta\alpha})$ for the symmetric part of a tensor.

Consider the operator $A$ and the 2-tensor $B$, defined by
$$A\indices{_{i\alpha}}{^{j\beta}}=2\epsilon\indices{^{jk}}{_i}I\indices{_{k\alpha}}{^{\beta}}, \quad B_{ij}=R_{kl\alpha\beta}\epsilon\indices{^{kl}}{_i}I\indices{_j}{^{\alpha\beta}}.
$$
Using the first Bianchi identity from \cite{IMV,IV}, Kunkel \cite{Kunk} shows that $A$ is invertible and
\begin{gather}\nonumber
\tilde{A}=A+\mathcal O_{(m)}; \quad \tilde A^{-1}=A^{-1}+\mathcal O_{(m)}; \\  \label{tilde}
\tilde{T}_{\alpha jk}\tilde{\epsilon}\indices{_i}{^{jk}}=T_{\alpha jk}\epsilon\indices{_i}{^{jk}}+A\indices{_{i\alpha}}{^{j\beta}}u_{\beta j}+\mathcal O_{(m-2)};
\\\nonumber
\tilde B_{(ij)}=B_{(ij)}+16nu_{(ij)}+\mathcal O_{(m-3)}.
\end{gather}
Defining the symmetric tensor $\mathcal Q$ with
\begin{equation}\label{tQ}
\begin{split}
\mathcal Q_{\alpha\beta}=L_{\alpha\beta}+\frac1{8(n+2)}Sg_{\alpha\beta};\quad
\mathcal Q_{\alpha i}=\mathcal Q_{i \alpha}=-(A^{-1})\indices{_{i\alpha}}{^{j\beta}}T_{\beta kl}\epsilon\indices{_j}{^{kl}};\quad
\mathcal Q_{ij}=-\frac1{16n}B_{(ij)},
\end{split}
\end{equation}
it follows from \eqref{tildeL} and \eqref{tilde}  that for $u\in\mathcal O_{(m)}$ the symmetric tensor $\mathcal Q$ changes as follows \cite{Kunk}
\begin{equation}\label{tildeQ}
\begin{split}
\tilde{\mathcal  Q}_{\alpha\beta}=\mathcal Q_{\alpha\beta}-u_{(\alpha\beta)}+\triangle ug_{\alpha\beta}+\mathcal O_{(m-1)};\quad
\tilde{\mathcal Q}_{i\alpha}=\mathcal Q_{i \alpha}-u_{i\alpha}+\mathcal O_{(m-2)}, \quad
\tilde{\mathcal Q}_{ij}=\mathcal Q_{ij}-u_{(ij)}+\mathcal O_{(m-3)}.
\end{split}
\end{equation}
In his main theorem Kunkel [\cite{Kunk}, Theorem~3.16] proves that
for any $q\in M$ and any $m\geq 2$ there is $u$ which is a
homogeneous polynomial of order $m$ in parabolic normal
coordinates $(x,t)$ such that all the symmetrized covariant
derivatives of $\tilde{\mathcal Q}$ with total order less than or
equal to $m$ vanish at the point $q\in M$, i.e. $\tilde{\mathcal
Q}_{(ab,C)}(q)=0$ if $o(abC)\leq m$. Such a parabolic normal
coordinates are called qc parabolic normal coordinates.

Using this normalization for $\mathcal Q$, \eqref{tor0}  and the
identities from \cite{IMV} and \cite{IV} Kunkel shows that in the
center $q$ of the qc parabolic normal coordinates the Ricci
tensor, scalar curvature, quaternionic contact torsion, the Ricci
type tensors and many of their covariant derivatives vanish,
\begin{thrm}[\cite{Kunk}, Theorem~3.17]\label{mainK}
Let $(M,\eta,g)$ be a qc manifold  for
which the symmetrized covariant derivatives of the tensor $\mathcal{Q}$ vanish to total order $4$
at a point $q\in M$. Then the following curvature and torsion terms vanish at $q$.
\begin{equation*}
\begin{split} S, \quad \tau_{\alpha\beta},\quad \mu_{\alpha\beta},\quad L_{\alpha\beta},\quad R_{\alpha\beta},\quad \rho_{i\alpha\beta},\quad \zeta_{i\alpha\beta},\quad \sigma_{i\alpha\beta},\quad T_{\alpha i \beta},\quad T_{ijk},\quad T_{\alpha jk},\\S_{,\beta},\quad
\mu\indices{_{\alpha\beta ,}}{^\alpha},\quad \tau\indices{_{\alpha\beta ,}}{^\alpha},\quad B_{ij},\quad S_{,i},\quad S\indices{_{,\alpha}}{^\alpha},\quad\tau\indices{_{\alpha\beta ,}}{^{\alpha\beta}},\quad \mu\indices{_{\alpha\beta ,}}{^{\alpha\beta}}, \quad R\indices{_{\gamma i\beta}}{^\gamma}{_{,\alpha}}I^{i\beta\alpha}.
\end{split}
\end{equation*}
\end{thrm}
We note that Kunkel proved Theorem~\ref{mainK} for a qc manifold of dimension bigger than $7,  (n>1)$  but a careful examination of his proof leads that Theorem~\ref{mainK} holds also for dimension $7$.

An important consequence of Theorem~\ref{mainK} is that at the center $q$ of the qc parabolic normal coordinates  the horizontal curvature is equal to its  $\frak{sp}(n)$-component, i.e. the following curvature identities hold at $q$:
\begin{multline}\label{Curvident}
\begin{aligned}
&R_{\alpha\beta\gamma\delta}(q)=-R_{\beta\alpha\gamma\delta}(q); \quad R_{\alpha\beta\gamma\delta}(q)=-R_{\alpha\beta\delta\gamma}(q);\quad R_{\alpha\beta\gamma\delta}(q)+R_{\beta\gamma\alpha\delta}(q)+R_{\gamma\alpha\beta\delta}(q)=0;\\&R_{\alpha\beta\gamma\delta}(q)=R_{\gamma\delta\alpha\beta}(q);\quad R_{\alpha\beta\gamma\delta}(q)I\indices{_i^{\delta\veps}}=R_{\alpha\beta\veps\delta}(q)I\indices{_i^{\delta\gamma}}.
\end{aligned}
\end{multline}
The first identity is clear,  the second one holds since the
Biquard connection preserves the metric. The third, fourth and
fifth  equalities are a  consequence of the first Bianchi identity
(see e.g. \cite[(3.2)]{IV}), \cite[Theorem~3.1]{IV},
\cite[Lemma~3.8]{IMV} and the fact that $S,\tau,\mu$ all vanish at
$q$ by Theorem~\ref{mainK}.

\subsection{Scalar polynomial invariants}
We recall here the definition of the notion "weight of a tensor", which plays a central role in our further considerations. First we remind the following
\begin{dfn}\label{weight}[\cite{Kunk} Definition~4.1.] Suppose $F$ is a homogeneous polynomial in $\{x^{\alpha},t^i\}$ whose coefficients are polynomial expressions in the curvature,
torsion and the covariant derivatives at $q.$ We define the \emph{weight} $w(F)$ recursively by
\begin{enumerate}[a)]
\item $w(T_{abc,D}(q))=o(bcD)-o(a);$
\item $w(R_{abcd,E}(q))=o(abcE)-o(d)=o(abE),$ since $c$ and $d$ always have the same order;
\item $w(F_1F_2)=w(F_1)+w(F_2);$
\item $w(g_{ab}(q))=w(g^{ab}(q))=w(I_{i\alpha\beta}(q))=w(\epsilon_{ijk}(q))=w(c)=0,$ where $c$ denotes an arbitrary constant, independent of the pseudohermitian structure;
\item $w(0)=m$ for all $m;$
\item if $w(F_A)=m$ for all $A,$ then $w(\sum_AF_Ax^A)=m.$
\end{enumerate}
\end{dfn}

The notion  "weight of a tensor" is an extension of the above definition. Namely, one says that a
 tensor $P$ has \emph{weight} $m$ and designate $w(P)=m,$ if its components with respect to the bases $\{X_a\}_{a=1}^{4n+3}$ and $\{\Xi^a\}_{a=1}^{4n+3}$ have weight $m.$
 Note that an arbitrary tensor $P$ can be decomposed in homogeneous parts and the components with respect to this bases are certain homogeneous polynomial in $\{x^{\alpha},t^i\}$.

Using the curvature and torsion identities from \cite{IMV,IV} Kunkel give in [\cite{Kunk}, Table~1] the following list of curvature and torsion terms of weight less then or equal to four:
\begin{itemize}
\item Weight 0: $g_{\alpha\beta}, g_{ij}, I_{i\alpha\beta},\epsilon_{ijk}$;
\item Weight 1: $T_{\alpha\beta\gamma}=0, T_{ij\alpha}=0$;
\item Weight 2: $T_{\alpha i\beta}, T_{ijk}, R_{\alpha\beta\gamma\delta}$;
\item Weight 3: $T_{\alpha ij}, T_{\alpha i\beta,\gamma}, R_{\alpha\beta\gamma\delta,\rho}, R_{\alpha i\beta\gamma}$;
\item Weight 4: $T_{\alpha ij,\beta}, T_{\alpha i\beta,\gamma\delta}, T_{\alpha i\beta,j}, R_{\alpha\beta\gamma\delta,\rho\sigma}, R_{\alpha\beta\gamma\delta,i}, R_{\alpha i\beta\gamma,\delta}, R_{ij\beta\gamma}$,
\end{itemize}
and shows in [\cite{Kunk}, Theorem~4.3] that  at the origin
of a qc parabolic  normal coordinates the only  tensors of weight at most four are the dimensional constants and the squared norm
of the qc conformal curvature tensor \eqref{qc-tensor}, namely we have
\begin{thrm}[\cite{Kunk}, Theorem~4.3]\label{normqc}
Let $(M,\eta,g,\mathbb{Q})$  be a qc manifold. Then, at the center $q\in M$ of the qc parabolic normal coordinates, the only invariant scalar quantities of weight no more than $4$ constructed
as polynomials from the invariants listed above are constants independent of
the qc structure and $||W^{qc}||^2$,
in particular, all other invariant scalar terms vanish at $q\in M$.
\end{thrm}
In what follows we use also the  next
\begin{conv}\label{conv1}\hfill\break\vspace{-15pt}
\begin{enumerate}[a)]
\item We denote by $(\eta=(\eta_1,\eta_2,\eta_3),g)$ the qc structure normalized according to Theorem~\ref{mainK}. The corresponding qc parabolic normal coordinates will be signified by $\{x^\alpha,t^i\}$.
\item We shall use $\{\xi_1,\ldots,\xi_{4n},R_1,R_2,R_3\}=\{\xi_a\}_{a=1}^{4n+3}$ to denote the special frame, corresponding to the contact form $\eta.$ The (dual) co-frame will be designated by $\{\theta^1,\ldots,\theta^{4n},\eta^1,\eta^2,\eta^3\}=\{\theta^a\}_{a=1}^{4n+3}$.
\item The index notations of the tensors will be used only with respect to the special frame $\{\xi_a\}_{a=1}^{4n+3}$ and the special co-frame $\{\theta^a\}_{a=1}^{4n+3}.$ For example, $A_{\alpha\beta}=A(\xi_\alpha,\xi_\beta), \quad B\indices{_{\alpha\beta}^\tgamma}=\theta^{\tgamma}(B(\xi_\alpha,\xi_\beta))$ and s.o.
\end{enumerate}
\end{conv}

\section{The asymptotic expansion of the qc Yamabe functional}\label{Asymptexp}
In order to find an asymptotic expansion of the qc Yamabe
functional we prove a number of lemmas.
\begin{lemma}\label{leftinvhom}
For the standard left-invariant frame and co-frame on  $\mathbf{G}(\mathbb{H})$ the next assertions hold:
\begin{enumerate}[a)]
\item The vector fields $X_\alpha$ are homogeneous of order $-1.$
\item The vector fields $T_i$ are homogeneous of order $-2.$
\item The $1-$forms $\Xi^\alpha$ are homogeneous of order $1.$
\item The $1-$forms $\Theta^i$ are homogeneous of order $2.$
\end{enumerate}
\end{lemma}
\begin{proof} To check a), take the Lie derivative of $X_\alpha$ with respect to the vector field $P$ .
For a smooth function $f$ we calculate $(\mathcal{L}_PX_\alpha)f=-X_\alpha f,$ i.e. $\mathcal{L}_PX_\alpha=-X_\alpha.$ Similarly, $\mathcal{L}_PT_i=-2T_i$ which proofs b).
Tho check c) and d), we use the Cartan formula $\mathcal{L}_X\omega=X\lrcorner d\omega+d(X\lrcorner\omega)$ for a vector field $X\in\Gamma(TM)$ and a differential form $\omega\in\Omega(M).$
Standard calculations lead to $\mathcal{L}_P\Xi^\alpha=\Xi^\alpha$ and $\mathcal{L}_P\Theta^i=2\Theta^i.$ Note that the last facts are implicitly mentioned in Proposition~\ref{recurconn}.
\end{proof}
\noindent The next lemma gives an information for the homogeneous parts of certain orders of the coordinate functions of the special frame $\{\xi_a\}_{a=1}^{4n+3}$ with respect to the
standard left-invariant vector fields $\{X_a\}_{a=1}^{4n+3}$ on $\mathbf{G}(\mathbb{H})$. 
\begin{lemma}\label{decompxi}
If $\xi_a=s_a^bX_b,$ then the following relations hold:
\begin{equation}\label{coorfun}
s_{\alpha(0)}^{\beta}=\delta_{\alpha}^{\beta},\quad s_{\alpha(0)}^{\talpha}=s_{\alpha(1)}^{\talpha}=s_{\talpha(0)}^{\alpha}=s_{\talpha(1)}^{\tbeta}=0,\quad s_{\talpha(0)}^{\tbeta}=\delta_{\talpha}^{\tbeta}.
\end{equation}
Moreover, for a natural number $m$, the next recursive formula is true:
\begin{equation}\label{coorrec}
s_{a(m+o(b)-o(a))}^b=-\sum_{i\geq 2}s_{a(m+o(c)-o(a)-i)}^c\theta_{(o(b)+i)}^b(X_c).
\end{equation}
\end{lemma}
\begin{proof}
First, we have $\xi_\alpha=s_{\alpha}^{\beta}X_\beta+s_{\alpha}^{\talpha}X_{\talpha}.$ We get by Proposition~\ref{Leftinv} and Lemma~\ref{leftinvhom} that
$\xi_{\alpha}\in\mathcal{O}_{(-1)},\quad s_{\alpha}^{\beta}X_\beta\in\mathcal{O}_{(-1)}$ and $s_{\alpha}^{\talpha}X_{\talpha}\in\mathcal{O}_{(-2)}$ (note that $s_a^b\in\mathcal{O}_{(0)}$).
Taking the homogeneous parts of order $-1$ and $-2$ in the above equality, we obtain that $s_{\alpha(0)}^{\beta}=\delta_{\alpha}^{\beta},\quad s_{\alpha(1)}^{\talpha}=0$ and $s_{\alpha(0)}^{\talpha}=0,$
respectively. Similarly, taking the homogeneous parts of order $-2$  and $-1$ in the equality $\xi_{\talpha}=s_{\talpha}^{\alpha}X_\alpha+s_{\talpha}^{\tbeta}X_{\tbeta}$
we get $s_{\talpha(0)}^{\tbeta}=\delta_{\talpha}^{\tbeta}$ and $s_{\talpha(0)}^{\alpha}=s_{\talpha(1)}^{\tbeta}=0,$ respectively, which proofs \eqref{coorfun}.

In order to prove \eqref{coorrec}  we take the homogeneous parts of order $m+o(b)-o(a)$ in  the equality $\delta_a^b=s_a^c\theta^b(X_c)$. We  separate the proof in two cases.
The first case appears when $m+o(b)-o(a)>0.$ Then we get
\begin{equation*}
\begin{split}
0=\delta_{a(m+o(b)-o(a))}^b=\sum_{i\geq0}s_{a(m+o(c)-o(a)-i)}^c[\theta^b(X_c)]_{(o(b)-o(c)+i)}, \quad \textnormal{i.e.} \\ s_{a(m+o(b)-o(a))}^b=-\sum_{i\geq2}s_{a(m+o(c)-o(a)-i)}^c\theta_{(o(b)+i)}^b(X_c).
\end{split}
\end{equation*}
 Note that the left-hand side of the last identity is just the term  in the sum that corresponds to $i=0,$ while the term that correlates to $i=1$ is equal to $0,$ by Proposition~\ref{recurconn}.

The case $m+o(b)-o(a)=0$ occurs only when $m=1, o(b)=1$ and $o(a)=2.$ Hence, we have $b=\alpha, a=\talpha$ and the left-hand side of \eqref{coorrec} becomes $s_{\talpha(0)}^{\alpha}=0,$ according to \eqref{coorfun}. Moreover, $s_{\talpha(o(c)-1-i)}^c=0$ for $i\geq2,$ and the right-hand side of \eqref{coorrec} becomes also $0$
which completes the proof.
\end{proof}

A crucial auxiliary result that help us to find an asymptotic expansion of the qc Yamabe functional is 
\begin{lemma}\label{Cruclem} The differential forms $\eta_{(m)}^i$ and $d\eta_{(m)}^i$ have weight $m-2.$ The vector fields $\xi_{\alpha(m)}$ have weight $m+1,$
while the vector fields $\xi_{\talpha(m)}$ have weight $m+2.$ Finally, the function $S_{(m)}$ has weight $m+2.$
\end{lemma}
\begin{proof}
We shall prove by induction on $k$ that  all the objects
\begin{equation}\label{objectsrec}
\begin{aligned}
&\eta_{(k+2)}^i,\quad d\eta_{(k+2)}^i,\quad\theta_{(k+1)}^{\alpha},\quad\omega_{a(k)}^{\phantom{c}b},\quad\xi_{a(k-o(a))},\quad s_{a(k+o(b)-o(a))}^b,\\ &T_{abc,A(k-o(A)-o(bc)+o(a))},\quad R_{abcd,A(k-o(A)-o(ab))}
\end{aligned}
\end{equation}
have weight $k.$

\emph{A) Base of the induction: $k=0$.} We have $\eta^i_{(2)}=1.\Theta^i$ by Proposition~\ref{recurconn} and $w(1)=0$ by Definition~\ref{weight}, so $w(\eta^i_{(2)})=0.$
In the same way, $w(d\eta_{(2)}^i)=w(\theta_{(1)}^{\alpha})=w(\omega_{a(0)}^{\phantom{c}b})=0.$ Proposition~\ref{Leftinv} and Lemma~\ref{leftinvhom} imply $\xi_{a(-o(a))}=1.X_a$
which combined with $w(1)=0$ yield $w(\xi_{a(-o(a))})=0.$ From \eqref{coorfun} we have $s_{\alpha(o(\beta)-o(\alpha))}^{\beta}=s_{\alpha(0)}^{\beta}=\delta_{\alpha}^{\beta}$,
and similarly for the other choices of the indices $a$ and $b$.
Hence, we get $w(s_{a(o(b)-o(a))}^{b})=0.$  Next, we have $T_{abc,A(-o(A)-o(bc)+o(a))}=0$ except the case when $A=\emptyset,\quad a=\talpha,\quad b=\alpha,\quad c=\beta,$
in which situation we obtain $T_{\talpha\alpha\beta(0)}=-2I_{\talpha\alpha\beta}=Const.$ Thus, $w(T_{abc,A(-o(A)-o(bc)+o(a))})=0.$ Finally, we obtain in the same manner that $w(R_{abcd,A(-o(A)-o(ab))})=0,$
which completes the base of the induction.

\emph{B) Inductive step.} Suppose that all the objects in
\eqref{objectsrec} have weight $k$ for $k\leq m.$ We are going to
prove it holds for $k=m+1.$ The first step is to check the
assertion for the torsion and the curvature when $A=\emptyset.$

First we have to show that $T_{abc(m+1-o(bc)+o(a))}$ has weight $m+1.$ Applying Proposition~\ref{Taylor}, we get
\begin{equation*}
T_{abc(m+1-o(bc)+o(a))}=\sum_A\frac{1}{(\#A)!}\Big(\frac{1}{2}\Big)^{o(A)-\#A}x^A(X_AT_{abc})|_q,
\end{equation*}
where the sum is taken over all multi-indices $A:
o(A)=m+1-o(bc)+o(a)$. We show that $X_AT_{abc}|_q$ has weight
$m+1$ for any multi-index $A.$ (Note that we use here the same
letter $A$ to denote the corresponding multi-index; we are doing
this now and later in order to avoid the excessive accumulation of
letters). For that purpose, we  will prove  that $\xi_AT_{abc}|_q$
has weight $m+1$ for any multi-index $A$ with the mentioned order.

We recall that if $\{T_{a_1\ldots a_r}\}$ are the components of a
tensor $T$ of type $(0,r)$ with respect to the frame
$\{\xi_a\}_{a=1}^{4n+3}$ then  the components of the covariant
derivative of $T$ along the vector field $\xi_b$ are given by
\begin{equation}\label{covarder}
T_{A,b}=\xi_bT_A-\sum_{i=1}^{r}\omega\indices{_{a_i}^c}(\xi_b)T_{a_1\ldots a_{i-1}ca_{i+1}\ldots a_r},\quad A=(a_1\ldots a_r).
\end{equation}
We introduce the  notation:
\begin{equation*}
P_{abc,Ad}:=\omega\indices{_a^e}(\xi_d)T_{ebc,A}+\omega\indices{_b^e}(\xi_d)T_{aec,A}+\omega\indices{_c^e}(\xi_d)T_{abe,A}+\sum_{i=1}^{r}\omega\indices{_{a_i}^e}(\xi_d)T_{abc,a_1\ldots a_{i-1}ea_{i+1}\ldots a_r},\quad A=(a_1\ldots a_r).
\end{equation*}
Taking into  account \eqref{covarder} and the above notation, it is not difficult to see that
\begin{equation}\label{covtor}
T_{abc,A}=\xi_AT_{abc}-\sum_{i=1}^r\xi_{A_i}(P_{abc,B_i}),
\end{equation}
where  $A_i:=(a_{i+1}\ldots a_r),\quad i=0,\ldots,r-1,\quad A_r:=\emptyset,\quad B_i:=(a_1\ldots a_i),\quad i=1,\ldots,r,\quad B_0:=\emptyset.$

The next step is to take in \eqref{covtor} the homogeneous parts of order $0$. At first, we shall prove that
\begin{equation}\label{weightP}
w\big(P_{abc,B_i(k-o(B_i)-o(bc)+o(a))}\big)=k,\quad i=1,\ldots,r,\quad k\leq m+1.
\end{equation}
We have
\begin{equation}\label{firstterm}
[\omega\indices{_a^d}(\xi_{a_i})T_{dbc,B_{i-1}}]_{(k-o(B_i)-o(bc)+o(a))}=\sum_{k_1,k_2}[\omega\indices{_a^d}(\xi_{a_i})]_{(k_1)}[T_{dbc,B_{i-1}}]_{(k_2)},
\end{equation}
 where $k_1, k_2$ satisfy  $k_1+k_2=k-o(B_i)-o(bc)+o(a).$
Since $k_1, k_2\geq 0,$ we get $k_1\leq k-o(B_i)-o(bc)+o(a)$ and $k_2\leq k-o(B_i)-o(bc)+o(a),$ i.e. we obtain the conditions
\begin{equation}\label{condind1}
k_1\leq m+1-o(B_i)-o(bc)+o(a),\quad k_2\leq m+1-o(B_i)-o(bc)+o(a).
\end{equation}
 To investigate the weight of the term $[\omega\indices{_a^d}(\xi_{a_i})]_{(k_1)},$ we decompose it in the sum
\begin{equation}\label{condind2}
[\omega\indices{_a^d}(\xi_{a_i})]_{(k_1)}=\sum_{l_1,l_2}\omega_{a(l_1)}^{\phantom{c}d}\xi_{a_i(l_2)},
\end{equation}
where $l_1$ and $l_2$ satisfy $l_1+l_2=k_1,\quad l_1\geq 2,\quad l_2\geq-o(a_i).$ These conditions and the first inequality in \eqref{condind1} give  $l_1=k_1-l_2\leq m+1-o(B_{i-1})-o(bc)+o(a)\leq m.$ The last inequality does not hold only in the trivial case when $B_{i-1}=\emptyset,\quad o(bc)=2,\quad o(a)=2.$ But this situation occurs only if $a\in\{4n+1,4n+2,4n+3\},\quad b,c\in\{1,\ldots,4n\},$ which implies $k_1,k_2\leq m$ and we can apply the inductive hypothesis to the terms that appears in the right-hand side of \eqref{firstterm} to obtain that the term in the left-hand side of \eqref{firstterm} has weight $k$. So, we can apply the inductive hypothesis for $\omega\indices{_a^d}$ to conclude that
\begin{equation}\label{weightom}
w(\omega_{a(l_1)}^{\phantom{c}d})=l_1.
\end{equation}
Moreover, the first inequality in \eqref{condind1} and the inequality $l_1\geq 2$ gives  $l_2=k_1-l_1\leq m-2,$ which allows  to apply the inductive hypothesis for $\xi_{a_i},$ i.e.
\begin{equation}\label{weightxi}
w(\xi_{a_i(l_2)})=l_2+o(a_i).
\end{equation}
Using \eqref{weightom}, \eqref{weightxi} and \eqref{condind2}, we conclude
\begin{equation}\label{weighox}
w\{[\omega\indices{_a^d}(\xi_{a_i})]_{(k_1)}\}=k_1+o(a_i).
\end{equation}
Furthermore, the second inequality in \eqref{condind1} implies $k_2+o(B_{i-1})+o(bc)-o(a)\leq m$ which together with the inductive hypothesis for the torsion yield
\begin{equation}\label{weightor}
w[(T_{dbc,B_{i-1}})_{(k_2)}]=k_2+o(B_{i-1})+o(bc)-o(d).
\end{equation}
Finally, we get from \eqref{firstterm}, \eqref{weighox} and \eqref{weightor} that $[\omega\indices{_a^d}(\xi_{a_i})T_{dbc,B_{i-1}}]_{(k-o(B_i)-o(bc)+o(a))}$ has weight $k.$ We obtain in the same manner that the homogeneous parts of order $k-o(B_i)-o(bc)+o(a)$ of the other terms in the definition of $P_{abc,B_i}$ all have weight $k,$ which proofs \eqref{weightP}.

To calculate the weight of $\xi_AT_{abc}|_q$ we consider $[\xi_{A_i}(P_{abc,B_i})]_{(0)},\quad i=1,\ldots,r,$ with
the decomposition:
\begin{multline}\label{new1}
[\xi_{A_i}(P_{abc,B_i})]_{(0)}=\sum_{k_0+k_{i+1}+\cdots+k_r=0}\xi_{a_r(k_r)}\ldots\xi_{a_{i+1}(k_{i+1})}(P_{abc,B_i})_{(k_o)}\\=\sum_{k_0+k_{i+1}+\cdots+k_r=0}\xi_{a_r(k_r+o(a_r)-o(a_r))}\ldots\xi_{a_{i+1}(k_{i+1}+o(a_{i+1})-o(a_{i+1}))}\times\\\times(P_{abc,B_i})_{(k_0-o(B_i)-o(bc)+o(a)+o(B_i)+o(bc)-o(a))}.
\end{multline}
The inequalities $0\leq k_{i+1}+o(a_{i+1})+\cdots+k_r+o(a_r)<o(A)\leq m+1$ imply $k_{i+1}+o(a_{i+1})\leq m,\ldots,k_r+o(a_r)\leq m,\quad k_0+o(B_i)+o(bc)-o(a)\leq o(A)+o(bc)-o(a)=m+1$
which  allow to apply the inductive hypothesis to the vector fields $\xi_{a_{i+1}},\ldots,\xi_{a_r}$
and  deduce from \eqref{weightP} and \eqref{new1} that

 \begin{multline}\label{weightxip}
 w\{[\xi_{A_i}(P_{abc,B_i})]_{(0)}\}=k_r+o(a_r)+\cdots+k_{i+1}+o(a_{i+1})+k_0+o(B_i)+o(bc)-o(a)\\=o(A)+o(bc)-o(a)=m+1.
 \end{multline}
 Look at the homogeneous parts of order zero in \eqref{covtor} taking into account \eqref{weightxip} and Definition~\ref{weight} to get
 \begin{equation}\label{weightxiA}
 w(\xi_AT_{abc}|_q)=m+1,\quad\textrm{where}\quad o(A)=m+1-o(bc)+o(a).
 \end{equation}
We have consecutively:

 \begin{multline}\label{weighXAT}
 \xi_AT_{abc}|_q=(\xi_{a_r}\ldots\xi_{a_1}T_{abc})_{(0)}\\=\sum_{k_0+k_1+\cdots+k_r=0}\xi_{a_r(k_r)}\ldots\xi_{a_1(k_1)}T_{abc(k_0)}=\xi_{a_r(-o(a_r))}\ldots\xi_{a_1(-o(a_1))}T_{abc(o(A))}\\+\sum_{\substack{k_0+k_1+\cdots+k_r=0\\k_0<o(A)}}\xi_{a_r(k_r)}\ldots\xi_{a_1(k_1)}T_{abc(k_0)}=X_{a_r}\ldots X_{a_1}T_{abc(o(A))}+\sum_{\substack{k_0+k_1+\cdots+k_r=0\\k_0<o(A)}}\xi_{a_r(k_r)}\ldots\xi_{a_1(k_1)}T_{abc(k_0)}\\=X_AT_{abc}|_q+\sum_{\substack{k_0+k_1+\cdots+k_r=0\\k_0<o(A)}}\xi_{a_r(k_r+o(a_r)-o(a_r))}\ldots\xi_{a_1(k_1+o(a_1)-o(a_1))}T_{abc(k_0)}.
 \end{multline}
We are interesting in the weight of the summands at the right-hand side of \eqref{weighXAT}. We may take  $k_0>0$ since the corresponding term is zero when $k_0=0$.
Hence $k_1+o(a_1)+\cdots+k_r+o(a_r)<o(A)=m+1-o(bc)+o(a)\leq m+1,$ i.e. $k_1+o(a_1)\leq m,\ldots,k_r+o(a_r)\leq m$ and  the inductive hypothesis applied  to the vector fields $\xi_{a_i}, i=1,\ldots,r,$  gives
\begin{equation}\label{wxi}
w(\xi_{a_i(k_i)})=k_i+o(a_i),\quad i=1,\ldots,r.
\end{equation}
Moreover, since $k_0<o(A)\leq m+1,$  the inductive hypothesis for $T_{abc}$ yields
\begin{equation}\label{wtabc}
w(T_{abc(k_0)})=k_0+o(bc)-o(a).
\end{equation}
Finally, we apply \eqref{weightxiA}, \eqref{wxi} and \eqref{wtabc} to  \eqref{weighXAT} to get
 \begin{equation}\label{weightXA}
 w(X_AT_{abc}|_q)=m+1,\quad\textrm{where}\quad o(A)=m+1-o(bc)+o(a)
 \end{equation} which completes the proof that $T_{abc(m+1-o(bc)+o(a))}$ has weight $m+1.$

The proof of the fact that $R_{abcd(m+1-o(ab))}$ has weight $m+1$ is similar and we omit it.

It follows from Proposition~\ref{recurconn} and the just proved  two results for the torsion and the curvature that $\eta^i_{(k+2)},\quad d\eta^i_{(k+2)},\quad\theta^{\alpha}_{(k+1)},\quad \omega_{a(k)}^{\phantom{c}b}$ have weight $k$ for $k=m+1.$

Now, we  check that $\xi_{a(k-o(a))}$ and $s_{a(k+o(b)-o(a))}^b$ have weight $k$ for $k=m+1.$ We have
\begin{multline}\label{weightxis}
\xi_{a(m+1-o(a))}\\=(s_a^bX_b)_{(m+1-o(a))}=s_{a(m+1-o(a)+o(b))}^bX_b=-\sum_{i\geq 2}s_{a(m+1+o(c)-o(a)-i)}^c\theta_{(o(b)+i)}^b(X_c)X_b,
\end{multline}
where we used Lemma~\ref{leftinvhom} and \eqref{coorrec} for the second and the third equality, respectively. It is clear that
\begin{equation}\label{weigs}
w(s_{a(m+1+o(c)-o(a)-i)}^c)=m+1-i,
\end{equation}
because  $m+1-i\leq m-1$ and the inductive hypothesis for $s_a^c.$
On the other hand, $m+1+o(c)-o(a)-i\geq 0,$ which implies $i\leq
m+1.$ (The case $i=m+2$  occurs only when $i=m+1+o(c)-o(a),\quad
o(c)=2,\quad o(a)=1,$ in which situation $s_{a(0)}^c=0$ and the
corresponding term in the sum of the right-hand side of
\eqref{weightxis} is zero.) Proposition~\ref{recurconn} and the
inductive hypothesis imply $w(\theta_{(o(b)+i)}^b)=i,$ which,
together with \eqref{weigs} and \eqref{weightxis} give
$w(\xi_{a(m+1-o(a))})=w(s_{a(m+1+o(b)-o(a))}^b)=m+1.$ This
completes the induction in the case $A=\emptyset.$

In order to finish the proof that all the objects in \eqref{objectsrec} have weight $k,$ we have to show that $T_{abc,A(k-o(A)-o(bc)+o(a))}$ and $R_{abcd,A(k-o(A)-o(ab))}$ have weight $k$. We shall do it for the torsion by a new induction on $\#A=s,\quad A=(a_1,\ldots,a_s).$ For the curvature the proof is the same.

\emph{a) Base of the induction: $s=1$.} Using  the results  for $\xi_{a_1}, T_{dbc}$, $\omega\indices{_a^d}$,  it is easy to see that
\begin{equation}\label{base1}
w[(\xi_{a_1}T_{abc})_{(l)}]=w\{[\omega\indices{_a^d}(\xi_{a_1})T_{dbc}]_{(l)}\}=l+o(a_1)+o(bc)-o(a).
\end{equation}
We get from  \eqref{base1} and  \eqref{covarder}  that $T_{abc,a_1(l)}$ has weight $l+o(a_1)+o(bc)-o(a),$ which finishes the base of the induction.

\emph{b) Inductive step.} Suppose  $w[(T_{abc,A})_{(l)}]=l+o(A)+o(bc)-o(a)$ for any multi-index $A: \#A=s, A=(a_1,\ldots,a_s).$ Using the inductive hypothesis and the results we already have for $\xi_a$ and $\omega\indices{_a^b}$ we obtain
\begin{equation*}\label{step1}
w[(\xi_dT_{abc,A})_{(l)}]=w\{[\omega\indices{_a^e}(\xi_d)T_{ebc,A}]_{(l)}\}=l+o(Ad)+o(bc)-o(a)
\end{equation*}
which together with the identity $T_{abc,Ad}=\xi_d(T_{abc,A})-P_{abc,Ad}$ show that $(T_{abc,Ad})_{(l)}$ has weight $l+o(Ad)+o(bc)-o(a)$.
The induction is completed.

Finally,  $S_{(m)}$ has weight $m+2$ since $S=g^{\alpha\beta}g^{\gamma\delta}R_{\alpha\gamma\delta\beta}$ which ends the proof of the Lemma.
\end{proof}
The next
result establishes   a relation between the  volume forms on $(M,\eta,g,Q)$ and
on $\mathbf{G}(\mathbb{H})$.
\begin{lemma}\label{Conectvol}
Let $Vol_{\eta}=\frac{1}{4^n}\eta^1\wedge\eta^2\wedge\eta^3\wedge(d\eta^1)^{2n}$ and $Vol_{\Theta}=\frac{1}{4^n}\Theta^1\wedge\Theta^2\wedge\Theta^3\wedge(d\Theta^1)^{2n}$
be the natural volume forms on the qc manifold $(M,\eta,g Q)$ and the quaternionic Heisenberg group $\mathbf{G}(\mathbb{H}),$ respectively. T
hen $$Vol_{\eta}=[1+v_2+v_3+v_4+O(\rho^5)]Vol_\Theta,$$where $v_s$ is a homogeneous polynomial of degree $s$ and weight $s, s=1,2,3,4,$ and $O(\rho^5)$ is a function in $\mathcal{O}_{(5)}.$
\end{lemma}
\begin{proof}
According to Proposition~\ref{recurconn}, Lemma~\ref{leftinvhom} and Lemma~\ref{Cruclem}, we have the following  decomposition of $\eta^i, i=1,2,3,$ with respect to the base $\{\Theta^i,dx^{\alpha}\}$ of $\Lambda^1(TM)$ 
\begin{multline}\label{decompeta}
\eta^i=\eta_{(2)}^i+\eta_{(4)}^i+\eta_{(5)}^i+\cdots=\Theta^i+\underbrace{P_{22j}^i\Theta^j+P_{32\alpha}^idx^{\alpha}}_{\eta_{(4)}^i}+\underbrace{P_{33j}^i\Theta^j+P_{43\alpha}^idx^{\alpha}}_{\eta_{(5)}^i}+\cdots\\=\Theta^i+[P_{221}^i+P_{331}^i+P_{441}^i+O(\rho^5)]\Theta^1+[P_{222}^i+P_{332}^i+P_{442}^i+O(\rho^5)]\Theta^2\\+[P_{223}^i+P_{333}^i+P_{443}^i+O(\rho^5)]\Theta^3+[P_{32\alpha}^i+P_{43\alpha}^i+O(\rho^5)]dx^{\alpha},
\end{multline}
where $P_{stu}^i$ is a homogeneous polynomial in $\{x^{\alpha},t^i\}$ of degree $s$ and weight $t.$

Similarly, using again Proposition~\ref{recurconn}, Lemma~\ref{leftinvhom} and Lemma~\ref{Cruclem}, we get the following  representation of $d\eta^i, i=1,2,3,$ with respect to the base $\{\Theta^i\wedge\Theta^j,\Theta^i\wedge dx^{\alpha},dx^{\alpha}\wedge dx^{\beta}\}$ of $\Lambda^2(TM)$
\begin{multline}\label{decompdeta}
d\eta^i=d\Theta^i+\underbrace{P_{02jk}^i\Theta^j\wedge\Theta^k+P_{12j\alpha}^i\Theta^j\wedge dx^{\alpha}+P_{22\alpha\beta}^idx^{\alpha}\wedge dx^{\beta}}_{d\eta_{(4)}^i}\\+\underbrace{P_{13jk}^i\Theta^j\wedge\Theta^k+P_{23j\alpha}^i\Theta^j\wedge dx^{\alpha}+P_{33\alpha\beta}^idx^{\alpha}\wedge dx^{\beta}}_{d\eta_{(5)}^i}+\cdots\\=d\Theta^i+[P_{02jk}^i+P_{13jk}^i+P_{24jk}^i+P_{35jk}^i+P_{46jk}^i+O(\rho^5)]\Theta^j\wedge\Theta^k+[P_{12j\alpha}^i+P_{23j\alpha}^i+P_{34j\alpha}^i+P_{45j\alpha}^i+O(\rho^5)]\Theta^j\wedge dx^{\alpha}\\+[P_{22\alpha\beta}^i+P_{33\alpha\beta}^i+P_{44\alpha\beta}^i+O(\rho^5)]dx^{\alpha}\wedge dx^{\beta},
\end{multline}
where $P_{stuv}^i$ is a homogeneous polynomial in $\{x^{\alpha},t^i\}$ of degree $s$ and weight $t.$

It is not difficult to see, using \eqref{decompeta} and \eqref{decompdeta}, that
\begin{equation*}
\eta^1\wedge\eta^2\wedge\eta^3\wedge(d\eta^1)^{2n}=[1+v_2+v_3+v_4+O(\rho^5)]\Theta^1\wedge\Theta^2\wedge\Theta^3\wedge(d\Theta^1)^{2n},
\end{equation*}
where $v_s$ is a homogeneous polynomial of degree $s$ and weight $s, s=1,2,3,4,$ which proves the lemma.
\end{proof}
The next essential step is to find an asymptotic expression of the qc Yamabe functional \eqref{Yamabefunc1} over a special  set of "test functions".

It is shown in  \cite{IMV,IMV2,IMV4} that the function
$F:=[(1+|p|^2)^2+|w|^2]^{-(n+1)}$ is an extremal for the qc Yamabe
functional $\Upsilon_{\Theta}$ on the quaternionic Heisenberg
group $\mathbf{G}(\mathbb{H})$. Here
$|p|^2:=(x^1)^2+\cdots+(x^{4n})^2,\quad
|w|^2:=(t^1)^2+(t^2)^2+(t^3)^2$. The function
$F^{\varepsilon,\lambda}:=\varepsilon^{\lambda}\delta_{1/{\varepsilon}}^*F,\quad
\varepsilon, \lambda\in\mathbb{R},\quad\varepsilon>0,$ is also an
extremal. Let us choose $\lambda=-2(n+1)$ and denote
$F^{\varepsilon}:=F^{\varepsilon,-2(n+1)}=\varepsilon^{2(n+1)}[(\varepsilon^2+|p|^2)^2+|w|^2]^{-(n+1)}.$
It is not difficult to see that the integral
$\int_{\mathbf{G}(\mathbb{H})}(F^{\varepsilon})^{2^*}\,Vol_{\Theta}$
is a constant independent of $\varepsilon$. Indeed, we have
consecutively:

\begin{multline*}
\int_{\mathbf{G}(\mathbb{H})}(F^{\varepsilon})^{2^*}\,Vol_{\Theta}=\int_{\mathbf{G}(\mathbb{H})}\varepsilon^{-4n-6}(\delta_{1/{\varepsilon}}^*F)^{2^*}\,Vol_{\Theta}=\int_{\mathbf{G}(\mathbb{H})}(\delta_{1/{\varepsilon}}^*F)^{2^*}\delta_{1/{\varepsilon}}^*\,(Vol_{\Theta})\\=\int_{\mathbf{G}(\mathbb{H})}\delta_{1/{\varepsilon}}^*(F^{2^*}\,Vol_{\Theta})=\int_{\mathbf{G}(\mathbb{H})}F^{2^*}\,Vol_{\Theta}.
\end{multline*}
\emph{The natural distance function} on  $\mathbf{G}(\mathbb{H})$ is defined in the coordinates $\{x^{\alpha},t^i\}$ by $\rho:=\sqrt[4]{|p|^4+|t|^2}=\sqrt[4]{[(x^1)^2+\cdots+(x^{4n})^2]^2+(t^1)^2+(t^2)^2+(t^3)^2}.$ \emph{The polar change} of the coordinates $\{x^{\alpha},t^i\}$ is defined by 
\begin{equation}\label{polarch}
\begin{aligned}
&x^1=\rho\xi^1,\ldots,x^{4n}=\rho\xi^{4n}, t^1=\rho^2\tau^1, t^2=\rho^2\tau^2, t^3=\rho^2\tau^3,\quad \textrm{where}\\
&[(\xi^1)^2+\cdots+(\xi^{4n})^2]^2+(\tau^1)^2+(\tau^2)^2+(\tau^3)^2=1.
\end{aligned}
\end{equation}
Let $\{x^{\alpha},t^i\}$ be  qc parabolic normal coordinates near
$q\in M$ for a contact form $\eta$. We may assume the definition
area to be $\{\rho<2k\}$ for a suitable $k>0.$ Following
\cite{JL}, we define the function $\psi\in Co^{\infty}(M)$ with
compact support in the set $\{\rho<2k\}$ and $\psi\equiv 1$ in the
set $\{\rho<k\}.$ The "test function" $f^{\varepsilon}$ is defined
by $f^{\varepsilon}:=\psi F^{\varepsilon}, \varepsilon>0,$ on the
set $\{\rho<2k\}.$ The main result of this section is contained in
the following
\begin{prop}\label{firstexp}
For the qc Yamabe functional \eqref{Yamabefunc1} over the test functions $f^{\varepsilon}$ the next asymptotic expression holds
\begin{equation}\label{expan}
\Upsilon_{\eta}(f^{\varepsilon})=\frac{b_0(n)+b_4(n)||\W||^2\varepsilon^4+O(\varepsilon^5)}{\big(a_0(n)+a_4(n)||\W||^2\varepsilon^4+O(\varepsilon^5)\big)^{2/2^*}},
\end{equation}
where $a_0(n), a_4(n), b_0(n)$ and $b_4(n)$ are some independent of the qc structure dimensional constants.
\end{prop}
\begin{proof}
We shall examine separately each of the integrals that appear in the expression 
\begin{equation}\label{Yamabefunc2A}
\Upsilon_{\eta}(f^{\varepsilon})=\frac{\int_M[4\frac{n+2}{n+1}|\nabla f^{\varepsilon}|_{\eta}^2+S(f^{\varepsilon})^2]\,Vol_{\eta}}{[\int_M(f^{\varepsilon})^{2^*}\,Vol_{\eta}]^{2/2^*}}.
\end{equation}

\emph{A)} We begin with $\int_M(f^{\varepsilon})^{2^*}\,Vol_{\eta}.$

If $\varphi$ is an integrable function on $\mathbf{G}(\mathbb{H})$ and $|\varphi|\leq C\Phi(\rho)$ for a constant $C$ and an integrable function $\Phi$ on $\mathbf{G}(\mathbb{H})$ that depends only on $\rho$ then the next formula holds
\begin{equation}\label{intO}
\int_{a<\rho<b}\varphi\,Vol_{\Theta}=O\big(\int_{a}^b\Phi(\rho)\rho^{4n+5}d\rho\big).
\end{equation}
Indeed, the polar change \eqref{polarch} of the coordinates in the volume form on $\mathbf{G}(\mathbb{H})$ yields  $\,Vol_{\Theta}=\rho^{4n+5}d\rho\wedge d\sigma,$ where $d\sigma$ is a $4n+2$--form  that depends only on $\xi^1,\ldots,\xi^{4n},\tau^1,\tau^2,\tau^3$ and \eqref{intO} follows.
We also have 
\begin{multline*}
F^{-\frac1{n+1}}=(1+|p|^2)^2+|w|^2=[1+(x^1)^2+\cdots+(x^{4n})^2]^2+(t^1)^2+(t^2)^2+(t^3)^2\\\geq1+\rho^4[(\xi^1)^2+\cdots+(\xi^{4n})^2]^2+\rho^4[(\tau^1)^2+(\tau^2)^2+(\tau^3)^2]=1+\rho^4\geq\tilde{C}(1+\rho)^4
\end{multline*}
for some positive constant $\tilde{C}$ which implies
\begin{equation}\label{modF}
|F|\leq C(1+\rho)^{-4(n+1)}.
\end{equation}
We get consecutively:

\begin{multline}\label{denominator}
\int_M(f^{\varepsilon})^{2^*}\,Vol_{\eta}=\int_{\rho<k}(F^{\varepsilon})^{2^*}[1+v_2+v_3+v_4+O(\rho^5)]\,Vol_{\Theta}\\+\int_{k<\rho<2k}\psi^{2^*}(F^{\varepsilon})^{2^*}[1+v_2+v_3+v_4+O(\rho^5)]\,Vol_{\Theta}\\=\int_{\rho<k/{\varepsilon}}(\delta_{\varepsilon}^*F^{\varepsilon})^{2^*}[1+{\varepsilon}^2v_2+\veps^3v_3+\veps^4v_4+O(\veps^5\rho^5)]\delta_\veps^*\,Vol_{\Theta}\\+\int_{k/{\veps}<\rho<2k/{\veps}}\delta_\veps^*\psi^{2^*}(\delta_\veps^*F^{\veps})^{2^*}[1+{\varepsilon}^2v_2+\veps^3v_3+\veps^4v_4+O(\veps^5\rho^5)]\delta_\veps^*\,Vol_{\Theta}\\=\int_{\rho<k/\veps}F^{2^*}[1+{\varepsilon}^2v_2+\veps^3v_3+\veps^4v_4+O(\veps^5\rho^5)]\,Vol_{\Theta}\\+\int_{k/\veps<\rho<2k/\veps}\delta_\veps^*\psi^{2^*}F^{2^*}[1+{\varepsilon}^2v_2+\veps^3v_3+\veps^4v_4+O(\veps^5\rho^5)]\,Vol_{\Theta}\\=\int_{\rho<k/\veps}F^{2^*}[1+{\varepsilon}^2v_2+\veps^3v_3+\veps^4v_4+O(\veps^5\rho^5)]\,Vol_{\Theta}+O\big(\int_{k/\veps<\rho<2k/\veps}F^{2^*}\,Vol_{\Theta}\big)\\=\int_{\mathbf{G}(\mathbb{H})}F^{2^*}(1+{\varepsilon}^2v_2+\veps^3v_3+\veps^4v_4)\,Vol_{\Theta}-\int_{\rho>k/\veps}F^{2^*}(1+{\varepsilon}^2v_2+\veps^3v_3+\veps^4v_4)\,Vol_{\Theta}+\int_{\rho<k/\veps}F^{2^*}O(\veps^5\rho^5)\,Vol_{\Theta}\\+O\big(\int_{k/\veps<\rho<2k/\veps}F^{2^*}\,Vol_{\Theta}\big)=\int_{\mathbf{G}(\mathbb{H})}F^{2^*}(1+{\varepsilon}^2v_2+\veps^3v_3+\veps^4v_4)\,Vol_{\Theta}\\+O\big(\int_{k/\veps}^{\infty}\sum_{i=0}^4\veps^i\rho^i(1+\rho)^{-4(2n+3)}\rho^{4n+5}d\rho\big)+O\big(\int_{0}^{k/\veps}\veps^5\rho^5(1+\rho)^{-4(2n+3)}\rho^{4n+5}d\rho\big)\\+O\big(\int_{k/\veps}^{2k/\veps}(1+\rho)^{-4(2n+3)}\rho^{4n+5}d\rho\big),
\end{multline}
where we used the definition of $f^{\veps}$ and
Lemma~\ref{Conectvol} for the first equality and the parabolic
dilation change of the variables for the second one.  The third
identity follows  from the definition of $F^{\veps}$. To get the
fourth one, we used the next chain of relations:
\begin{multline*}
\big|\int_{k/\veps<\rho<2k/\veps}\delta_\veps^*\psi^{2^*}F^{2^*}[1+{\varepsilon}^2v_2+\veps^3v_3+\veps^4v_4+O(\veps^5\rho^5)]\,Vol_{\Theta}\big|\\\leq\int_{k/\veps<\rho<2k/\veps}C_0F^{2^*}|1+{\varepsilon}^2v_2+\veps^3v_3+\veps^4v_4+O(\veps^5\rho^5)|\,Vol_{\Theta}\\\leq\int_{k/\veps<\rho<2k/\veps}C_0F^{2^*}[1+\veps^2\rho^2+\veps^3\rho^3+\veps^4\rho^4+O(\veps^5\rho^5)]\,Vol_{\Theta}\\\leq\int_{k/\veps<\rho<2k/\veps}C_0F^{2^*}[\underbrace{1+(2k)^2+(2k)^3+(2k)^4+O\big((2k)^5\big)}_{\textrm{convergent series}}]\,Vol_{\Theta}=C_1\int_{k/\veps<\rho<2k/\veps}F^{2^*}\,Vol_{\Theta},
\end{multline*}
where $C_0$ and $C_1$ are some suitable positive constants and $k$ is chosen sufficient small. Hence,
$$\int_{k/\veps<\rho<2k/\veps}\delta_\veps^*\psi^{2^*}F^{2^*}[1+{\varepsilon}^2v_2+\veps^3v_3+\veps^4v_4+O(\veps^5\rho^5)]\,Vol_{\Theta}=O\big(\int_{k/\veps<\rho<2k/\veps}F^{2^*}\,Vol_{\Theta}\big).$$
The fifth identity in \eqref{denominator} is clear  while the sixth one is obtained from \eqref{intO} and \eqref{modF}.

For the first term at the right-hand side of \eqref{denominator} we have $$\int_{\mathbf{G}(\mathbb{H})}F^{2^*}(\underbrace{1}_{=:v_0}+{\varepsilon}^2v_2+\veps^3v_3+\veps^4v_4)\,Vol_{\Theta}=c_0+c_2\veps^2+c_3\veps^3+c_4\veps^4,$$ where $c_i:=\int_{\mathbf{G}(\mathbb{H})}F^{2^*}v_i\,Vol_{\Theta}, i=0,2,3,4,$ is a quaternionic contact invariant scalar quantity of weight $i.$ It follows by Theorem~\ref{normqc} that $c_2=c_3=0$ and therefore
\begin{equation}\label{firsttermInt}
\int_{\mathbf{G}(\mathbb{H})}F^{2^*}(1+{\varepsilon}^2v_2+\veps^3v_3+\veps^4v_4)\,Vol_{\Theta}=a_0(n)+a_4(n)||\W||^2\veps^4,
\end{equation}
where $a_0(n)$ and $a_4(n)$ are some dimensional constants, independent on the qc structure.

Finally, we are interesting in the last three expressions at the right-hand side of \eqref{denominator}.
We have for $i=0,\ldots,4$ that $\rho^{4n+5+i}(1+\rho)^{-8n-12}\leq\rho^{-6}$ and consequently,
\begin{equation}\label{firstO}
\int_{k/\veps}^{\infty}\sum_{i=0}^4\veps^i\rho^i(1+\rho)^{-4(2n+3)}\rho^{4n+5}d\rho=O(\veps^5).
\end{equation}
In a similar way, we obtain
\begin{equation}\label{secondO}
\int_{0}^{k/\veps}\veps^5\rho^5(1+\rho)^{-4(2n+3)}\rho^{4n+5}d\rho=O(\veps^5), \quad
\int_{k/\veps}^{2k/\veps}(1+\rho)^{-4(2n+3)}\rho^{4n+5}d\rho=O(\veps^5).
\end{equation}
A substitution of \eqref{firsttermInt}, \eqref{firstO} and \eqref{secondO}  into \eqref{denominator} gives
\begin{equation}\label{denexp}
\int_M(f^{\veps})^{2^*}\,Vol_{\eta}=a_0(n)+a_4(n)||\W||^2\veps^4+O(\veps^5).
\end{equation}

\emph{B)} We continue with $\int_M|\nabla f^{\veps}|_{\eta}^2\,Vol_{\eta}.$
It follows directly by the decomposition of $\xi_a$ determined in Lemma~\ref{decompxi}, the definition of the function $f^{\veps}$ and Lemma~\ref{Conectvol} that
\begin{multline}\label{numerator1}
\int_M|\nabla f^{\veps}|_{\eta}^2\,Vol_{\eta}\\=\int_M\sum_{\alpha=1}^{4n}(\xi_{\alpha}f^{\veps})^2\,Vol_{\eta}=\int_{\rho<k}\sum_{\alpha=1}^{4n}s_\alpha^as_{\alpha}^b(X_aF^{\veps})(X_bF^{\veps})[1+v_2+v_3+v_4+O(\rho^5)]\,Vol_{\Theta}\\+\int_{k<\rho<2k}\sum_{\alpha=1}^{4n}s_{\alpha}^as_{\alpha}^b[(X_a\psi)F^{\veps}+\psi X_a(F^{\veps})][(X_b\psi)F^{\veps}+\psi X_b(F^{\veps})][1+v_2+v_3+v_4+O(\rho^5)]\,Vol_{\Theta}=:I_1+I_2.
\end{multline}
We begin with $I_1.$
At first, note that the following representation holds $$\sum_{\alpha=1}^{4n}s_{\alpha}^as_{\alpha}^b=\sum_{\alpha=1}^{4n}\sum_{m=0}^{\infty}\sum_{k_1+k_2=m}s_{\alpha(k_1+o(a)-1)}^as_{\alpha(k_2+o(b)-1)}^b=:\sum_{m=0}^{\infty}v_{m}^{ab},$$
where $v_{m}^{ab}, m=0,1,\ldots,\infty,$ is a homogeneous polynomial of degree $m+o(ab)-2$ and weight $m.$ We have
\begin{multline}\label{numerator1A}
I_1=\int_{\rho<k}(\sum_{m=0}^{\infty}v_{m}^{ab})(X_aF^\veps)(X_bF^\veps)[1+v_2+v_3+v_4+O(\rho^5)]\,Vol_{\Theta}\\=\int_{\rho<k}[w_0^{ab}+w_1^{ab}+w_2^{ab}+w_3^{ab}+w_4^{ab}+O(\rho^{3+o(ab)})](X_aF^\veps)(X_bF^\veps)\,Vol_{\Theta},
\end{multline}
where the homogeneous polynomial $w_i^{ab}, i=0,\ldots,4,$ is formed from the polynomials $v_m^{ab}$,  $v_j$ and is of degree $i+o(ab)-2$ and of weight $i.$
Moreover, it follows directly by the definition of $F^\veps$ that
\begin{equation}\label{deltaXa}
\delta_\veps^*(X_aF^\veps)=\veps^{-2(n+1)-o(a)}X_aF.
\end{equation}
Another fact we shall need is the inequality
\begin{equation}\label{XaFineq}
|X_aF|\leq {C (1+\rho)^{-4(n+1)-o(a)}},
\end{equation}
where $C$ is some positive constant.\footnote{We shall use again
the letters $C$ and $\tilde{C}$ to denote some (positive)
constants. We are doing this now and later in order to avoid the
excessive accumulation of different letters and indices.} In order
to check \eqref{XaFineq}, we suppose firstly that
$a\in\{1,\ldots,4n\}.$
Using the definitions of the function $F$ and the vector field $X_a,$ we obtain
\begin{multline*}
X_aF=-4(n+1)\big\{[1+(x^1)^2+\cdots+(x^{4n})^2]^2+(t^1)^2+(t^2)^2+(t^3)^2\big\}^{-n-2}\times\\\times\big\{[1+(x^1)^2+\cdots+(x^{4n})^2]x^a+I\indices{^1_{\beta
a}}x^\beta t^1+I\indices{^2_{\beta a}}x^\beta
t^2+I\indices{^3_{\beta a}}x^\beta t^3\big\}.
\end{multline*}
The polar change \eqref{polarch} in the above identity gives
$|X_aF|\leq\tilde{C}(1+\rho)^{-4(n+2)}(\rho+\rho^3)$ for some
positive constant $\tilde{C},$ yielding $|X_aF|\leq
C(1+\rho)^{-4n-5}=C(1+\rho)^{-4(n+1)-o(a)},$ where $C$ is  a
positive constant. Thus \eqref{XaFineq} is proved  for
$a\in\{1,\ldots,4n\}.$ The case  $a\in\{4n+1,4n+2,4n+3\}$ is
considered in a similar way.

We return to \eqref{numerator1A},  make the parabolic dilations change $\delta_{\veps}(x^{\alpha},t^i)=(\veps x^\alpha,\veps^2t^i)$  and use \eqref{deltaXa}  to get
\begin{multline}\label{numerator1AA}
I_1=\int_{\rho<k/\veps}\sum_{i=0}^4\veps^iw_i^{ab}(X_aF)(X_bF)\,Vol_{\Theta}+\int_{\rho<k/\veps}O(\veps^5\rho^{3+o(ab)})(X_aF)(X_bF)\,Vol_{\Theta}\\=\int_{\mathbf{G}(\mathbb{H})}\sum_{i=0}^4\veps^iw_i^{ab}(X_aF)(X_bF)\,Vol_{\Theta}-\int_{\rho>k/\veps}\sum_{i=0}^4\veps^iw_i^{ab}(X_aF)(X_bF)\,Vol_{\Theta}\\+\int_{\rho<k/\veps}O(\veps^5\rho^{3+o(ab)})(X_aF)(X_bF)\,Vol_{\Theta}.
\end{multline}
Now we shall examine the three integrals in the right-hand side of \eqref{numerator1AA}. We have
\begin{equation*}
\int_{\mathbf{G}(\mathbb{H})}\sum_{i=0}^4\veps^iw_i^{ab}(X_aF)(X_bF)\,Vol_{\Theta}=d_0+\veps d_1+\veps^2 d_2+\veps^3 d_3+\veps^4 d_4,
\end{equation*}
where $d_i:=\int_{\mathbf{G}(\mathbb{H})}w_i^{ab}(X_aF)(X_bF)\,Vol_{\Theta}, i=0,\ldots,4,$ is a quaternionic contact invariant scalar quantity of weight $i$. In the same manner as in \eqref{firsttermInt} we obtain  by Theorem~\ref{normqc}  that
\begin{equation}\label{numerator1AAA}
\int_{\mathbf{G}(\mathbb{H})}\sum_{i=0}^4\veps^iw_i^{ab}(X_aF)(X_bF)\,Vol_{\Theta}=\tilde{b}_0(n)+\tilde{b}_4(n)||\W||^2\veps^4,
\end{equation}
where $\tilde{b}_0(n)$ and $\tilde{b}_4(n)$ are some dimensional constants, independent of the qc structure.

We get using \eqref{intO} and \eqref{XaFineq} that
\begin{multline*}
\big|\int_{\rho>k/\veps}\sum_{i=0}^4\veps^iw_i^{ab}(X_aF)(X_bF)\,Vol_{\Theta}\big|\leq C\int_{\rho>k/\veps}\sum_{i=0}^4\veps^i\rho^{i+o(ab)-2}(1+\rho)^{-8(n+1)-o(ab)}\,Vol_{\Theta}\\=O\big[\sum_{i=0}^4\veps^i\int_{k/\veps}^\infty\big(\rho^{4n+3+i+o(ab)}(1+\rho)^{-8(n+1)-o(ab)}\big)d\rho\big]
\end{multline*}
which combined with the inequality $\rho^{4n+3+i+o(ab)}(1+\rho)^{-8(n+1)-o(ab)}\leq\rho^{-6+i}$ lead  to
\begin{equation}\label{numerator1AAB}
\int_{\rho>k/\veps}\sum_{i=0}^4\veps^iw_i^{ab}(X_aF)(X_bF)\,Vol_{\Theta}=O(\veps^5).
\end{equation}
To handle the last integral in the right-hand side of \eqref{numerator1AA}, we  use \eqref{XaFineq} to get
\begin{multline*}
\big|\int_{\rho<k/\veps}O(\veps^5\rho^{3+o(ab)})(X_aF)(X_bF)\,Vol_{\Theta}\big|\\\leq C\int_{\rho<k/\veps}\veps^5(\rho^{3+o(ab)}+\veps\rho^{4+o(ab)}+\cdots)(1+\rho)^{-8(n+1)-o(ab)}\,Vol_{\Theta},
\end{multline*}
which together with  $\rho^{3+o(ab)}+\veps\rho^{4+o(ab)}+\cdots\leq\tilde{C}\rho^{3+o(ab)}$ for $\rho<\frac{k}{\veps}$, $k$ sufficiently small, and \eqref{intO} give  

\begin{equation*}
\int_{\rho<k/\veps}O(\veps^5\rho^{3+o(ab)})(X_aF)(X_bF)\,Vol_{\Theta}=O\big[\veps^5\int_{0}^{k/\veps}\rho^{4n+8+o(ab)}(1+\rho)^{-8(n+1)-o(ab)}d\rho\big].
\end{equation*}
Now, it is easy to see after some standard analysis  that
\begin{equation}\label{numerator1AAC}
\int_{\rho<k/\veps}O(\veps^5\rho^{3+o(ab)})(X_aF)(X_bF)\,Vol_{\Theta}=O(\veps^5).
\end{equation}

We substitute \eqref{numerator1AAA}, \eqref{numerator1AAB} and \eqref{numerator1AAC} into \eqref{numerator1AA} to obtain
\begin{equation}\label{numerator1AALast}
I_1=\tilde{b}_0(n)+\tilde{b}_4(n)||\W||^2\veps^4+O(\veps^5).
\end{equation}

To manage  the integral $I_2$ in \eqref{numerator1}, we note  that for  $k<\rho<2k $ the following inequalities hold $|s_\alpha^a||s_\alpha^b|\leq M_{\alpha}^{ab}\quad\textrm{and}\quad |1+v_2+v_3+v_4+O(\rho^5)|\leq N$ for suitable constants $M_{\alpha}^{ab}, N$. Then we have
\begin{equation*}
|I_2|\leq\int_{k<\rho<2k}\sum_{\alpha=1}^{4n}C_{\alpha}^{ab}\big(|F^\veps|^2+|X_aF^\veps||F^\veps|+|X_bF^\veps||F^\veps|+|X_aF^\veps||X_bF^\veps|\big)\,Vol_{\Theta}
\end{equation*} for some positive constants $C_{\alpha}^{ab}$. The latter  yields
\begin{multline}\label{numerator1B}
I_2=O\big(\int_{k<\rho<2k}|F^\veps|^2\,Vol_{\Theta}\big)\\+O\big(\sum_{a=1}^{4n+3}\int_{k<\rho<2k}|X_aF^\veps||F^\veps|\,Vol_{\Theta}\big)+O\big(\sum_{a,b=1}^{4n+3}\int_{k<\rho<2k}|X_aF^\veps||X_bF^\veps|\,Vol_{\Theta}\big).
\end{multline}
We  examine the three integrals  in the right-hand side of \eqref{numerator1B}. We begin with
$$\int_{k<\rho<2k}|F^\veps|^2\,Vol_{\Theta}=\veps^2\int_{k/\veps<\rho<2k/\veps}|F|^2\,Vol_{\Theta}=\veps^2O\big(\int_{k/\veps}^{2k/\veps}(1+\rho)^{-8(n+1)}\rho^{4n+5}d\rho\big),$$
where we applied the parabolic dilation change of the variables to obtain the first equality,  \eqref{intO} and \eqref{modF} to get the second one. The latter  together with the inequality $(1+\rho)^{-8(n+1)}\rho^{4n+5}\leq\rho^{-4}$ lead  to
\begin{equation}\label{numerator1BA}
\int_{k<\rho<2k}|F^\veps|^2\,Vol_{\Theta}=O(\veps^5).
\end{equation}
Similarly,  we obtain  using \eqref{intO}, \eqref{modF}, \eqref{deltaXa} and \eqref{XaFineq}  the next two relations
\begin{equation}\label{numerator1BB}
\int_{k<\rho<2k}|X_aF^\veps||F^\veps|\,Vol_{\Theta}\\=O(\veps^5), \quad
\int_{k<\rho<2k}|X_aF^\veps||X_bF^\veps|\,Vol_{\Theta}=O(\veps^5).
\end{equation}
A substitution of \eqref{numerator1BA}, \eqref{numerator1BB}  into \eqref{numerator1B} gives
\begin{equation}\label{newI2}I_2=O(\veps^5)
\end{equation}
which combined  with \eqref{numerator1AALast} imply
\begin{equation}\label{numerator1Last}
\int_M|\nabla f^{\veps}|_{\eta}^2\,Vol_{\eta}=\tilde{b}_0(n)+\tilde{b}_4(n)||\W||^2\veps^4+O(\veps^5).
\end{equation}
 \emph{C)} Finally, it remains to investigate the integral $\int_MS(f^\veps)^2\,Vol_{\eta}$ that appears in \eqref{Yamabefunc2A}. We have
\begin{multline}\label{numerator2}
\int_MS(f^\veps)^2\,Vol_{\eta}=\int_{\QH}[S_{(0)}+S_{(1)}+S_{(2)}+O(\rho^3)](f^\veps)^2[1+v_2+O(\rho^3)]\,Vol_{\Theta}\\=\int_{\rho<k}[S_{(0)}+S_{(1)}+S_{(2)}+O(\rho^3)](F^\veps)^2[1+v_2+O(\rho^3)]\,Vol_{\Theta}\\+\int_{k<\rho<2k}[S_{(0)}+S_{(1)}+S_{(2)}+O(\rho^3)]\psi^2(F^\veps)^2[1+v_2+O(\rho^3)]\,Vol_{\Theta}=:J_1+J_2,
\end{multline}
where we used Lemma~\ref{Conectvol} to get the first equality and the definition of $f^\veps$ to obtain the second one.

We examine separately the integrals $J_1$ and $J_2$ in \eqref{numerator2}. For the integral $J_1$ we have
\begin{multline}\label{numerator2A}
J_1=\int_{\rho<k/\veps}[S_{(0)}\veps^2+S_{(1)}\veps^3+(S_{(0)}v_2+S_{(2)})\veps^4+O(\rho^3\veps^5)]F^2\,Vol_{\Theta}\\=\int_{\QH}[S_{(0)}\veps^2+S_{(1)}\veps^3+(S_{(0)}v_2+S_{(2)})\veps^4]F^2\,Vol_{\Theta}-\int_{\rho>k/\veps}[S_{(0)}\veps^2+S_{(1)}\veps^3+(S_{(0)}v_2+S_{(2)})\veps^4]F^2\,Vol_{\Theta}\\+\int_{\rho<k/\veps}O(\rho^3\veps^5)F^2\,Vol_{\Theta}.
\end{multline}
We handle the three integrals that stay in the right-hand side of \eqref{numerator2A}.
We have firstly
\begin{equation*}
\int_{\QH}[S_{(0)}\veps^2+S_{(1)}\veps^3+(S_{(0)}v_2+S_{(2)})\veps^4]F^2\,Vol_{\Theta}=\veps^2e_2+\veps^3e_3+\veps^4e_4,
\end{equation*}
where $e_2:=\int_{\QH}S_{(0)}F^2\,Vol_{\Theta},\quad e_3:=\int_{\QH}S_{(1)}F^2\,Vol_{\Theta}\quad$and$\quad e_4:=\int_{\QH}(S_{(0)}v_2+S_{(2)})F^2\,Vol_{\Theta}$
are some quaternionic contact invariant scalar quantities of weight $2, 3$ and $4,$ respectively. It follows from Theorem~\ref{normqc} that $e_2=e_3=0.$
To show $e_4=0$ we use that $S_{(0)}=S|_q=0$ by Theorem~\ref{mainK}. Furthermore, $$S_{(2)}=\sum_{o(A)=2}\frac{1}{(\#A)!}\Big(\frac{1}{2}\Big)^{o(A)-\#A}x^A(X_AS)|_q=\sum_{o(A)=2}\frac{1}{(\#A)!}\Big(\frac{1}{2}\Big)^{o(A)-\#A}x^A(S_{,A})|_q=\sum_{\alpha,\beta}\frac{1}{2}x^\alpha x^\beta(S_{,\alpha\beta})|_q.$$
Indeed,    Proposition~\ref{Taylor} implies  the first equality. The second one follows from the fact that $S\in\mathcal{O}_{1}$ and Lemma~\ref{derK}. The third identity is a consequence of the  equality $S_{,\talpha}|_q=0, \talpha\in\{4n+1,4n+2,4n+3\},$ see Theorem~\ref{mainK}. So, we obtain $e_4=S_{,\alpha\beta}|_qc^{\alpha\beta},$ where $c^{\alpha\beta}:=\frac{1}{2}\int_{\QH}x^\alpha x^\beta F^2\,Vol_{\Theta}\quad$ is a quaternionic contact scalar invariant quantity of weight $0.$ But the only qc scalar invariant quantities that can be formed by a complete contraction of $S_{,\alpha\beta}$ are $S\indices{_{,\alpha}^\alpha}$ and $S_{,\talpha}$ (see \cite{Kunk}, p. 30), which are  zero at $q$ because of Theorem~\ref{mainK}. Thus  $e_4=0$ and we  conclude that
\begin{equation}\label{numerator2AA}
\int_{\QH}[S_{(0)}\veps^2+S_{(1)}\veps^3+(S_{(0)}v_2+S_{(2)})\veps^4]F^2\,Vol_{\Theta}=0.
\end{equation}

To deal with the second integral in the right-hand side of \eqref{numerator2A}, we are based on the equality
\begin{equation*}
\int_{\rho>k/\veps}[S_{(0)}\veps^2+S_{(1)}\veps^3+(S_{(0)}v_2+S_{(2)})\veps^4]F^2\,Vol_{\Theta}=O\big[\int_{k/\veps}^\infty\big(\sum_{i=0}^2\rho^i\veps^{i+2}\big)(1+\rho)^{-8(n+1)}\rho^{4n+5}d\rho\big],
\end{equation*}
following from \eqref{intO} and \eqref{modF}. This limiting behavior together with  $(1+\rho)^{-8(n+1)}\rho^{4n+5+i}\leq\rho^{i-4}$ yield
\begin{equation}\label{numerator2AB}
\int_{\rho>k/\veps}[S_{(0)}\veps^2+S_{(1)}\veps^3+(S_{(0)}v_2+S_{(2)})\veps^4]F^2\,Vol_{\Theta}=O(\veps^5).
\end{equation}

To handle the third integral in the right-hand side of \eqref{numerator2A}, we observe at first the equality
\begin{equation*}
\int_{\rho<k/\veps}O(\rho^3\veps^5)F^2\,Vol_{\Theta}=O\big[\veps^5\int_0^\infty\rho^{4n+8}(1+\rho)^{-8(n+1)}d\rho\big],
\end{equation*}
which is a direct consequence from \eqref{intO} and \eqref{modF}. It is not difficult to see that this identity together with the inequality $\rho^{4n+8}(1+\rho)^{-8(n+1)}$ imply
\begin{equation}\label{numerator2AC}
\int_{\rho<k/\veps}O(\rho^3\veps^5)F^2\,Vol_{\Theta}=O(\veps^5).
\end{equation}

We substitute \eqref{numerator2AA}, \eqref{numerator2AB} and \eqref{numerator2AC} in \eqref{numerator2A} to obtain
\begin{equation}\label{numerator2ALast}
J_1=O(\veps^5).
\end{equation}

It remains to investigate the integral $J_2$ from
\eqref{numerator2}. We apply the parabolic dilations change of the
variables to $J_2$ and use \eqref{intO} and \eqref{modF} to get
after some standard calculations
\begin{equation*}
J_2=O\big[\int_{k/\veps}^{2k/\veps}\big(\underbrace{1+\veps\rho+\veps^2\rho^2+O(\veps^3\rho^3)}_{\leq Const}\big)\veps^2(1+\rho)^{-8(n+1)}\rho^{4n+5}d\rho\big]=O\big[\veps^2\int_{k/\veps}^{2k/\veps}(1+\rho)^{-8(n+1)}\rho^{4n+4}d\rho\big].
\end{equation*}
The latter combined with the inequality $(1+\rho)^{-8(n+1)}\rho^{4n+5}<\rho^{-4}$ yields $\quad J_2=O(\veps^5)$ which combined with \eqref{numerator2} and  \eqref{numerator2ALast} implies
\begin{equation}\label{numerator2Last}
\int_MS(f^\veps)^2\,Vol_{\eta}=O(\veps^5).
\end{equation}

We substitute \eqref{denexp}, \eqref{numerator1Last} and
\eqref{numerator2Last} in \eqref{Yamabefunc2A} and set
$b_0(n):=4\frac{n+2}{n+1}\tilde{b}_0(n)$ and
$b_4(n):=4\frac{n+2}{n+1}\tilde{b}_4(n),$ which ends the proof of
Proposition~\ref{firstexp}.
\end{proof}

\section{Explicit evaluation of constants}\label{explicit}

The aim of our investigations in this section is to calculate explicitly the constants $a_0(n), a_4(n), b_0(n)$ and $b_4(n)$ that stay at the right-hand side of \eqref{expan}. We begin with an algebraic lemma.
\begin{lemma}\label{algebraic}
If $\omega=\omega_{\alpha\beta}dx^{\alpha}\wedge dx^{\beta}$ is a
$2$-form, then\footnote{Here we shall use  the comma between the
lower indices to separate them, not to designate the covariant
differentiation. Actually, the sense of the comma will be clear
from the context.}:
\begin{enumerate}[a)]
\item $2n\Theta^1\wedge\Theta^2\wedge\Theta^3\wedge\omega\wedge(d\Theta^1)^{2n-1}=tr\omega\Theta^1\wedge\Theta^2\wedge\Theta^3\wedge(d\Theta^1)^{2n},$ where $tr\omega:=\omega_{12}+\cdots+\omega_{4n-1,4n}=-\frac{1}{2}I\indices{^1_{\alpha\beta}}\omega_{\alpha\beta};$
\item $8n(2n-1)\Theta^1\wedge\Theta^2\wedge\Theta^3\wedge\omega^2\wedge(d\Theta^1)^{2n-2}=(I\indices{^1_{\alpha\beta}}I\indices{^1_{\gamma\delta}}+2I\indices{^1_{\alpha\delta}}I\indices{^1_{\beta\gamma}})\omega_{\alpha\beta}\omega_{\gamma\delta}\Theta^1\wedge\Theta^2\wedge\Theta^3\wedge(d\Theta^1)^{2n}.$
\end{enumerate}
\end{lemma}
\begin{proof}
The claim a) follows from the next equalities
\begin{multline*}
\Theta^1\wedge\Theta^2\wedge\Theta^3\wedge\omega\wedge(d\Theta^1)^{2n-1}=\Theta^1\wedge\Theta^2\wedge\Theta^3\wedge(\omega_{\alpha\beta}dx^\alpha\wedge dx^\beta)\wedge(-I\indices{^1_{\alpha\beta}}dx^\alpha\wedge dx^\beta)^{2n-1}\\=\Theta^1\wedge\Theta^2\wedge\Theta^3\wedge(2\omega_{12}dx^1\wedge dx^2+\cdots+2\omega_{4n-1,4n}dx^{4n-1}\wedge dx^{4n})\wedge(2dx^1\wedge dx^2+\cdots+2dx^{4n-1}\wedge dx^{4n})^{2n-1}\\=\frac{1}{2n}tr\omega\Theta^1\wedge\Theta^2\wedge\Theta^3\wedge(d\Theta^1)^{2n}.
\end{multline*}

To handle b), we note  that the following formula holds:
\begin{equation}\label{algebrA}
8n(2n-1)\Theta^1\wedge\Theta^2\wedge\Theta^3\wedge dx^l\wedge dx^{l+1}\wedge dx^m\wedge dx^{m+1}\wedge(d\Theta^1)^{2n-2}=\Theta^1\wedge\Theta^2\wedge\Theta^3\wedge(d\Theta^1)^{2n},
\end{equation}
where $l, m\in\{1,3,\ldots,4n-1\}, l\neq m.$ We obtain consecutively
\begin{multline*}
\Theta^1\wedge\Theta^2\wedge\Theta^3\wedge\omega^2\wedge(d\Theta^1)^{2n-2}=\Theta^1\wedge\Theta^2\wedge\Theta^3\wedge(\omega_{\alpha\beta}\omega_{\gamma\delta}dx^{\alpha}\wedge dx^{\beta}\wedge dx^{\gamma}\wedge dx^{\delta})\wedge(d\Theta^1)^{2n-2}\\=\Theta^1\wedge\Theta^2\wedge\Theta^3\wedge[8(\omega_{12}\omega_{34}-\omega_{13}\omega_{24}+\omega_{14}\omega_{23})dx^1\wedge dx^2\wedge dx^3\wedge dx^4\\+\cdots+8(\omega_{4n-3,4n-2}\omega_{4n-1,4n}-\omega_{4n-3,4n-1}\omega_{4n-2,4n}+\omega_{4n-3,4n}\omega_{4n-2,4n-1})dx^{4n-3}\wedge dx^{4n-2}\wedge dx^{4n-1}\wedge dx^{4n}]\wedge\\\wedge(d\Theta^1)^{2n-2}=\frac{1}{n(2n-1)}(\omega_{12}\omega_{34}-\omega_{13}\omega_{24}+\omega_{14}\omega_{23}+\cdots+\omega_{4n-3,4n-2}\omega_{4n-1,4n}-\omega_{4n-3,4n-1}\omega_{4n-2,4n}\\+\omega_{4n-3,4n}\omega_{4n-2,4n-1})\Theta^1\wedge\Theta^2\wedge\Theta^3\wedge(d\Theta^1)^{2n}=\frac{1}{8n(2n-1)}(I\indices{^1_{\alpha\beta}}I\indices{^1_{\gamma\delta}}-I\indices{^1_{\alpha\gamma}}I\indices{^1_{\beta\delta}}+I\indices{^1_{\alpha\delta}}I\indices{^1_{\beta\gamma}})\omega_{\alpha\beta}\omega_{\gamma\delta}\times\\\times\Theta^1\wedge\Theta^2\wedge\Theta^3\wedge(d\Theta^1)^{2n}=\frac{1}{8n(2n-1)}(I\indices{^1_{\alpha\beta}}I\indices{^1_{\gamma\delta}}+2I\indices{^1_{\alpha\delta}}I\indices{^1_{\beta\gamma}})\omega_{\alpha\beta}\omega_{\gamma\delta}\Theta^1\wedge\Theta^2\wedge\Theta^3\wedge(d\Theta^1)^{2n},
\end{multline*}
where we used \eqref{algebrA} to get the third equality in the above chain. This completes the proof of the lemma.
\end{proof}

\begin{conv}\label{conv2}
From now on we shall use, similarly to the CR case \cite{JL}, the
notation $A\equiv B$ to designate the equivalence of the
expressions $A$ and $B$ modulo terms that contain the torsion or
the curvature or their covariant derivatives except
$R_{\alpha\beta\gamma\delta}(q),$ as well as modulo terms of
weight bigger than $4.$ We shall also use this notation when we
omit expressions containing powers of $\veps$ that lead through
the computations to powers different from $\veps^0$ and $\veps^4.$
The reason is because  we know by \eqref{expan} what kind of terms
appear  in the asymptotic expression of the qc Yamabe functional
over the test functions.
\end{conv}

We continue with a lemma which is crucial for the explicit evaluation of constants.

\begin{lemma}\label{ExplicitLemma}
For the test function $f^\veps$ defined in Section~\ref{Asymptexp} and the parabolic dilation $\delta_\veps$ in qc parabolic normal coordinates defined in Section~\ref{qcnormalcoord}  the following formulas hold:
\begin{multline}\label{dilatchange}
\begin{aligned}
&\delta_\veps^*[(f^\veps)^{2^*}\eta^1\wedge\eta^2\wedge\eta^3\wedge(d\eta^1)^{2n}]\equiv F^{2^*}[1+\veps^4\chi(x)]\Theta^1\wedge\Theta^2\wedge\Theta^3\wedge(d\Theta^1)^{2n};\\
&\delta_\veps^*[|\nabla f^\veps|_\eta^2\eta^1\wedge\eta^2\wedge\eta^3\wedge(d\eta^1)^{2n}]\equiv 16(n+1)^2[(1+|p|^2)^2+|w|^2]^{-2n-4}\big\{[(1+|p|^2)^2+|w|^2]|p|^2\\&+\frac{17}{720}\veps^4\sum_{i=1}^3I\indices{^i_{\beta\alpha}}I\indices{^i_{\xi\theta}}R\indices{_{\delta\alpha\iota}^\gamma}(q)R\indices{_{\eta\gamma\zeta}^\theta}(q)x^\beta x^\delta x^\iota x^\zeta x^\eta x^\xi(t^i)^2\big\}[1+\veps^4\chi(x)]\Theta^1\wedge\Theta^2\wedge\Theta^3\wedge(d\Theta^1)^{2n},
\end{aligned}
\end{multline}
where $\chi(x)$ is a homogeneous function of order $4$ defined below in \eqref{defchi}.
\end{lemma}
\begin{proof}
We begin with the relations
\begin{equation}\label{recurconnA}
\begin{aligned}
&\eta^i_{(2)}=\Theta^i;\quad\eta^i_{(3)}=0;\quad\eta^i_{(m)}\equiv\frac{1}{m}(t^j\omega\indices{_j^i}-2I\indices{^i_{\alpha\beta}}x^\alpha\theta^\beta)_{(m)},\quad m\geq 4;\\&\theta_{(1)}^\alpha=\Xi^\alpha;\quad\theta^\alpha_{(2)}=0;\quad\theta^\alpha_{(m)}\equiv\frac{1}{m}(x^\beta\omega\indices{_\beta^\alpha})_{(m)},\quad m\geq 3;\\&\omega_{a(1)}^{\phantom{c}b}=0;\quad\omega_{a(m)}^{\phantom{c}b}\equiv\frac{1}{m}(R\indices{_{\alpha\beta a}^b}x^\alpha\theta^\beta)_{(m)}\equiv\frac{1}{m}R\indices{_{\alpha\beta a}^b}(q)x^\alpha\theta^\beta_{(m-1)},\quad m\geq 2,
\end{aligned}
\end{equation}
which are consequences of Proposition~\ref{recurconn}.

We examine the homogeneous parts of certain orders.
 For the lowest possible order we have the  formula
\begin{equation}\label{hompartA}
[\eta^1\wedge\eta^2\wedge\eta^3\wedge(d\eta^1)^{2n}]_{(4n+6)}=\Theta^1\wedge\Theta^2\wedge\Theta^3\wedge(d\Theta^1)^{2n}.
\end{equation}
We see by a simple induction over $m$ in \eqref{recurconnA} that only the even-degree homogeneous parts  of $\eta^1, \eta^2$ and $\eta^3$ are non-zero which implies
\begin{equation}\label{hompartB}
[\eta^1\wedge\eta^2\wedge\eta^3\wedge(d\eta^1)^{2n}]_{(4n+7)}\equiv[\eta^1\wedge\eta^2\wedge\eta^3\wedge(d\eta^1)^{2n}]_{(4n+9)}\equiv 0.
\end{equation}

The situation with the terms $[\eta^1\wedge\eta^2\wedge\eta^3\wedge(d\eta^1)^{2n}]_{(4n+8)}$ and $[\eta^1\wedge\eta^2\wedge\eta^3\wedge(d\eta^1)^{2n}]_{(4n+10)}$ is more complicated. We have for the first one the decomposition
\begin{multline}\label{decompfirst}
[\eta^1\wedge\eta^2\wedge\eta^3\wedge(d\eta^1)^{2n}]_{(4n+8)}=\eta^1_{(2)}\wedge\eta^2_{(2)}\wedge\eta^3_{(4)}\wedge(d\Theta^1)^{2n}+\eta^1_{(2)}\wedge\eta^2_{(4)}\wedge\eta^3_{(2)}\wedge(d\Theta^1)^{2n}\\+\eta^1_{(4)}\wedge\eta^2_{(2)}\wedge\eta^3_{(2)}\wedge(d\Theta^1)^{2n}+2n\eta^1_{(2)}\wedge\eta^2_{(2)}\wedge\eta^3_{(2)}\wedge(d\eta^1)_{(4)}\wedge(d\Theta^1)^{2n-1}.
\end{multline}
We get by \eqref{recurconnA} and some simple calculations that
\begin{equation}\label{etahom}
\eta^i_{(4)}\equiv-\frac{1}{12}I\indices{^i_{\alpha\beta}}R\indices{_{\delta\zeta\gamma}^\beta}(q)x^\alpha x^\gamma x^\delta dx^\zeta,\quad i=1, 2, 3,
\end{equation}
which leads  to the relations
\begin{equation}\label{decompfirstA}
\eta^1_{(2)}\wedge\eta^2_{(2)}\wedge\eta^3_{(4)}\wedge(d\Theta^1)^{2n}\equiv\eta^1_{(2)}\wedge\eta^2_{(4)}\wedge\eta^3_{(2)}\wedge(d\Theta^1)^{2n}\equiv\eta^1_{(4)}\wedge\eta^2_{(2)}\wedge\eta^3_{(2)}\wedge(d\Theta^1)^{2n}\equiv 0.
\end{equation}
Furthermore, we obtain by \eqref{etahom}
\begin{equation}\label{detahom}
(d\eta^i)_{(4)}=d(\eta^i_{(4)})\equiv-\frac{1}{12}[I\indices{^i_{\alpha\beta}}R\indices{_{\delta\zeta\gamma}^\beta}(q)x^\gamma x^\delta+I\indices{^i_{\gamma\beta}}R\indices{_{\delta\zeta\alpha}^\beta}(q)x^\gamma x^\delta+I\indices{^i_{\delta\beta}}R\indices{_{\alpha\zeta\gamma}^\beta}(q)x^\delta x^\gamma]dx^\alpha\wedge dx^\zeta
\end{equation}
which gives  by  straightforward calculations  the formula for the trace of $d\eta^1_{(4)}$:
\begin{equation*}
tr(d\eta^1_{(4)})=-\frac12I\indices{^1_{\alpha\beta}}(d\eta^1_{(4)})_{\alpha\beta}\equiv-\frac{1}{24}R_{\delta\gamma}(q)x^\delta x^\gamma+\frac{n}{6}I\indices{^1_\gamma^\beta}\zeta_{1\delta\beta}(q)x^\gamma x^\delta-\frac{n}{6}I\indices{^1_{\delta}^\beta}\sigma_{1\gamma\beta}(q)x^\gamma x^\delta,
\end{equation*}
where the tensors $\zeta_{1\delta\beta}$ and
$\sigma_{1\gamma\beta}$ are defined in
\eqref{CurvatureContractAlmost}. The latter combined with
Theorem~\ref{mainK} yields $ tr(d\eta^1_{(4)})\equiv 0$ which
together with Lemma~\ref{algebraic}, a) lead  to $
2n\eta^1_{(2)}\wedge\eta^2_{(2)}\wedge\eta^3_{(2)}\wedge(d\eta^1)_{(4)}\wedge(d\Theta^1)^{2n-1}\equiv
0.$

Substitute the latter and \eqref{decompfirstA} into \eqref{decompfirst} to get
\begin{equation}\label{hompartC}
[\eta^1\wedge\eta^2\wedge\eta^3\wedge(d\eta^1)^{2n}]_{(4n+8)}\equiv 0.
\end{equation}
We continue with the investigation of the term $[\eta^1\wedge\eta^2\wedge\eta^3\wedge(d\eta^1)^{2n}]_{(4n+10)}.$ We have the decomposition
\begin{multline}\label{decompsecond}
[\eta^1\wedge\eta^2\wedge\eta^3\wedge(d\eta^1)^{2n}]_{(4n+10)}=\eta^1_{(2)}\wedge\eta^2_{(2)}\wedge\eta^3_{(2)}\wedge[(d\eta^1)^{2n}]_{(4n+4)}\\+[\eta^1\wedge\eta^2\wedge\eta^3]_{(8)}\wedge[(d\eta^1)^{2n}]_{(4n+2)}+[\eta^1\wedge\eta^2\wedge\eta^3]_{(10)}\wedge[(d\eta^1)^{2n}]_{(4n)},
\end{multline}
and examine each of the terms  in the right-hand side of \eqref{decompsecond}. The first one is decomposed as follows:
\begin{multline}\label{decompsecondA}
\eta^1_{(2)}\wedge\eta^2_{(2)}\wedge\eta^3_{(2)}\wedge[(d\eta^1)^{2n}]_{(4n+4)}\\=2n\Theta^1\wedge\Theta^2\wedge\Theta^3\wedge(d\eta^1)_{(6)}\wedge(d\Theta^1)^{2n-1}+n(2n-1)\Theta^1\wedge\Theta^2\wedge\Theta^3\wedge(d\eta^1_{(4)})^2\wedge(d\Theta^1)^{2n-2}.
\end{multline}
To handle the first term in the right-hand side of \eqref{decompsecondA}, we use that
\begin{equation}\label{etahomsix}
\eta_{(6)}^i\equiv-\frac{1}{360}I\indices{^i_{\alpha\beta}}R\indices{_{\delta\theta\gamma}^\beta}(q)R\indices{_{\eta\iota\zeta}^\theta}(q)x^\alpha x^\gamma x^\delta x^\zeta x^\eta dx^\iota,\quad i=1, 2, 3,
\end{equation}
following  easily from \eqref{recurconnA}. The formula \eqref{etahomsix} yields  the equivalence
\begin{multline*}
(d\eta^i)_{(6)}=d(\eta^i_{(6)})\equiv-\frac{1}{360}\big(I\indices{^i_{\alpha\beta}}R\indices{_{\delta\theta\gamma}^\beta}(q)R\indices{_{\eta\sigma\zeta}^\theta}(q)+I\indices{^i_{\gamma\beta}}R\indices{_{\delta\theta\alpha}^\beta}(q)R\indices{_{\eta\sigma\zeta}^\theta}(q)+I\indices{^i_{\delta\beta}}R\indices{_{\alpha\theta\gamma}^\beta}(q)R\indices{_{\eta\sigma\zeta}^\theta}(q)\\+I\indices{^i_{\zeta\beta}}R\indices{_{\delta\theta\gamma}^\beta}(q)R\indices{_{\eta\sigma\alpha}^\theta}(q)+I\indices{^i_{\eta\beta}}R\indices{_{\delta\theta\gamma}^\beta}(q)R\indices{_{\alpha\sigma\zeta}^\theta}(q)\big)x^\gamma x^\delta x^\zeta x^\eta dx^\alpha\wedge dx^\sigma,\quad i=1, 2, 3,
\end{multline*}
which together with Theorem~\ref{mainK} and some standard computations give
\begin{equation*}
tr(d\eta^1_{(6)})\equiv\frac{1}{720}\big(R\indices{_{\delta\theta\gamma}^\beta}(q)R\indices{_{\eta\beta\zeta}^\theta}(q)+I^{1\alpha\sigma}I\indices{^1_{\gamma\beta}}R\indices{_{\delta\theta\alpha}^\beta}(q)R\indices{_{\eta\sigma\zeta}^\theta}(q)+I^{1\alpha\sigma}I\indices{^1_{\delta\beta}}R\indices{_{\alpha\theta\gamma}^\beta}(q)R\indices{_{\eta\sigma\zeta}^\theta}(q)\big)x^\gamma x^\delta x^\zeta x^\eta.
\end{equation*}
The last  formula and Lemma~\ref{algebraic}, a) imply
\begin{multline}\label{decompsecondAA}
2n\Theta^1\wedge\Theta^2\wedge\Theta^3\wedge(d\eta^1)_{(6)}\wedge(d\Theta^1)^{2n-1}\equiv\frac{1}{720}\big(R\indices{_{\delta\theta\gamma}^\beta}(q)R\indices{_{\eta\beta\zeta}^\theta}(q)+I^{1\alpha\sigma}I\indices{^1_{\gamma\beta}}R\indices{_{\delta\theta\alpha}^\beta}(q)R\indices{_{\eta\sigma\zeta}^\theta}(q)\\+I^{1\alpha\sigma}I\indices{^1_{\delta\beta}}R\indices{_{\alpha\theta\gamma}^\beta}(q)R\indices{_{\eta\sigma\zeta}^\theta}(q)\big)x^\gamma x^\delta x^\zeta x^\eta\Theta^1\wedge\Theta^2\wedge\Theta^3\wedge(d\Theta^1)^{2n}.
\end{multline}
In order to manipulate the second object in the right-hand side of
\eqref{decompsecondA}, we use Lemma~\ref{algebraic}, b), the
equivalence \eqref{detahom} and  Theorem~\ref{mainK}. After some
long but standard calculations we establish the equivalence
\begin{multline}\label{decompsecondAB}
n(2n-1)\Theta^1\wedge\Theta^2\wedge\Theta^3\wedge(d\eta^1_{(4)})^2\wedge(d\Theta^1)^{2n-2}\\\equiv-\frac{1}{576}\big(R\indices{_{\delta\alpha\gamma}^\beta}(q)R\indices{_{\eta\beta\zeta}^\alpha}(q)+2I^{1\alpha\theta}I\indices{^1_{\zeta\iota}}R\indices{_{\delta\theta\gamma}^\beta}(q)R\indices{_{\eta\beta\alpha}^\iota}(q)+2I^{1\alpha\theta}I\indices{^1_{\eta\iota}}R\indices{_{\delta\theta\gamma}^\beta}(q)R\indices{_{\alpha\beta\zeta}^\iota}(q)\\+I^{1\alpha\kappa}I^{1\lambda\theta}I\indices{^1_{\gamma\beta}}I\indices{^1_{\zeta\iota}}R\indices{_{\delta\theta\alpha}^\beta}(q)R\indices{_{\eta\kappa\lambda}^\iota}(q)+I^{1\alpha\kappa}I^{1\lambda\theta}I\indices{^1_{\delta\beta}}I\indices{^1_{\eta\iota}}R\indices{_{\alpha\theta\gamma}^\beta}(q)R\indices{_{\lambda\kappa\zeta}^\iota}(q)\\+2I^{1\alpha\kappa}I^{1\lambda\theta}I\indices{^1_{\gamma\beta}}I\indices{^1_{\eta\iota}}R\indices{_{\delta\theta\alpha}^\beta}(q)R\indices{_{\lambda\kappa\zeta}^\iota}(q)\big)x^\gamma x^\delta x^\zeta x^\eta\Theta^1\wedge\Theta^2\wedge\Theta^3\wedge(d\Theta^1)^{2n}.
\end{multline}
We continue with the second term  in the right-hand side of \eqref{decompsecond}, which  can be decomposed as follows
\begin{multline}\label{newlabel}
[\eta^1\wedge\eta^2\wedge\eta^3]_{(8)}\wedge[(d\eta^1)^{2n}]_{(4n+2)}=2n\eta^1_{(4)}\wedge\eta^2_{(2)}\wedge\eta^3_{(2)}\wedge(d\eta^1)_{(4)}\wedge(d\Theta^1)^{2n-1}\\+2n\eta^1_{(2)}\wedge\eta^2_{(4)}\wedge\eta^3_{(2)}\wedge(d\eta^1)_{(4)}\wedge(d\Theta^1)^{2n-1}+2n\eta^1_{(2)}\wedge\eta^2_{(2)}\wedge\eta^3_{(4)}\wedge(d\eta^1)_{(4)}\wedge(d\Theta^1)^{2n-1}.
\end{multline}
According to \eqref{etahom} and \eqref{detahom} the forms $\eta^i_{(4)}$ and $(d\eta^i)_{(4)}$ have no $dt^j$ term  which together with \eqref{newlabel} imply
\begin{equation}\label{decompsecondB}
[\eta^1\wedge\eta^2\wedge\eta^3]_{(8)}\wedge[(d\eta^1)^{2n}]_{(4n+2)}\equiv 0.
\end{equation}
We decompose the last term in the right-hand  side of \eqref{decompsecond} in a similar manner, namely
\begin{multline*}
[\eta^1\wedge\eta^2\wedge\eta^3]_{(10)}\wedge[(d\eta^1)^{2n}]_{(4n)}=\eta^1_{(2)}\wedge\eta^2_{(4)}\wedge\eta^3_{(4)}\wedge(d\Theta^1)^{2n}+\eta^1_{(4)}\wedge\eta^2_{(2)}\wedge\eta^3_{(4)}\wedge(d\Theta^1)^{2n}+\eta^1_{(4)}\wedge\eta^2_{(4)}\wedge\eta^3_{(2)}\wedge(d\Theta^1)^{2n}\\+\eta^1_{(2)}\wedge\eta^2_{(2)}\wedge\eta^3_{(6)}\wedge(d\Theta^1)^{2n}+\eta^1_{(2)}\wedge\eta^2_{(6)}\wedge\eta^3_{(2)}\wedge(d\Theta^1)^{2n}+\eta^1_{(6)}\wedge\eta^2_{(2)}\wedge\eta^3_{(2)}\wedge(d\Theta^1)^{2n}.
\end{multline*}
Since the forms $\eta^i_{(4)}$ and $\eta^i_{(6)}$ do not contain $dt^j$ term by \eqref{etahom} and \eqref{etahomsix}, the last  identity gives the relation
\begin{equation}\label{decompsecondC}
[\eta^1\wedge\eta^2\wedge\eta^3]_{(10)}\wedge[(d\eta^1)^{2n}]_{(4n)}\equiv 0.
\end{equation}

Now we get by \eqref{decompsecond}, \eqref{decompsecondA}, \eqref{decompsecondAA}, \eqref{decompsecondAB}, \eqref{decompsecondB} and \eqref{decompsecondC} the equivalence
\begin{equation}\label{decompThesecond}
[\eta^1\wedge\eta^2\wedge\eta^3\wedge(d\eta^1)^{2n}]_{(4n+10)}\equiv\chi(x)\Theta^1\wedge\Theta^2\wedge\Theta^3\wedge(d\Theta^1)^{2n},
\end{equation}
where the homogeneous function $\chi$ is defined by the equality
\begin{multline}\label{defchi}
\chi(x)=\chi(x^1,\ldots,x^{4n}):=\Big(-\frac{1}{2880}R\indices{_{\delta\alpha\gamma}^\beta}(q)R\indices{_{\eta\beta\zeta}^\alpha}(q)-\frac{1}{480}I^{1\alpha\theta}I\indices{^1_{\zeta\iota}}R\indices{_{\delta\theta\gamma}^\beta}(q)R\indices{_{\eta\beta\alpha}^\iota}(q)\\-\frac{1}{480}I^{1\alpha\theta}I\indices{^1_{\eta\iota}}R\indices{_{\delta\theta\gamma}^\beta}(q)R\indices{_{\alpha\beta\zeta}^\iota}(q)-\frac{1}{576}I^{1\alpha\kappa}I^{1\lambda\theta}I\indices{^1_{\gamma\beta}}I\indices{^1_{\zeta\iota}}R\indices{_{\delta\theta\alpha}^\beta}(q)R\indices{_{\eta\kappa\lambda}^\iota}(q)\\-\frac{1}{576}I^{1\alpha\kappa}I^{1\lambda\theta}I\indices{^1_{\delta\beta}}I\indices{^1_{\eta\iota}}R\indices{_{\alpha\theta\gamma}^\beta}(q)R\indices{_{\lambda\kappa\zeta}^\iota}(q)-\frac{1}{288}I^{1\alpha\kappa}I^{1\lambda\theta}I\indices{^1_{\gamma\beta}}I\indices{^1_{\eta\iota}}R\indices{_{\delta\theta\alpha}^\beta}(q)R\indices{_{\lambda\kappa\zeta}^\iota}(q)\Big)x^\gamma x^\delta x^\zeta x^\eta.
\end{multline}

The relations \eqref{hompartA}, \eqref{hompartB}, \eqref{hompartC} and \eqref{decompThesecond} lead  to the equivalence\footnote{Note that we consider the homogeneous parts of $\eta^1\wedge\eta^2\wedge\eta^3\wedge(d\eta^1)^{2n}$ up to order $4n+10$ since the higher order terms  belong to $O(\veps^5)$ in the expressions after the parabolic dilation change of the variables.}
\begin{equation*}
\eta^1\wedge\eta^2\wedge\eta^3\wedge(d\eta^1)^{2n}\equiv[1+\chi(x)]\Theta^1\wedge\Theta^2\wedge\Theta^3\wedge(d\Theta^1)^{2n},
\end{equation*}
which in turn gives
\begin{equation}\label{volchanvar}
\delta_{\veps}^*[\eta^1\wedge\eta^2\wedge\eta^3\wedge(d\eta^1)^{2n}]\equiv\veps^{4n+6}[1+\veps^4\chi(x)]\Theta^1\wedge\Theta^2\wedge\Theta^3\wedge(d\Theta^1)^{2n}.
\end{equation}
It follows by the very definition of the test function $f^\veps$
that $
\delta^*_{\veps}(f^\veps)^{2^*}=\veps^{-4n-6}\delta^*_{\veps}(\psi^{2^*})F^{2^*},
$ which together with \eqref{volchanvar} give the first formula in
\eqref{dilatchange}. We omit the multiplier
$\delta^*_{\veps}(\psi^{2^*})$ since we know from
\eqref{firsttermInt} and \eqref{denexp} that it appears only in
the $O(\veps^5)$--part of the denominator of  the asymptotic
expansion \eqref{expan}.

In order to prove the second formula in \eqref{dilatchange}, we need to find the effect of the parabolic dilation change of the variables on the squared norm
\begin{equation}\label{squarednorm} |\nabla f^\veps|^2_{\eta}:=\sum_{\alpha=1}^{4n}(\xi_\alpha f^\veps)^2=\sum_{\alpha=1}^{4n}(s_\alpha^aX_af^\veps)(s_\alpha^bX_bf^\veps).
\end{equation}
We continue with the computations we commenced in Lemma~\ref{decompxi}. We have
\begin{equation}\label{homsone}
s_{\alpha(1)}^\beta=-\sum_{i\geq 2}s_{\alpha(o(a)-i)}^a\theta^\beta_{(1+i)}(X_a)=0,
\end{equation}
where the first identity follows from \eqref{coorrec} and the second one--from \eqref{coorfun}. For the term $s_{\alpha(2)}^\beta$ we obtain
\begin{equation}\label{homssecond}
s_{\alpha(2)}^\beta=-\sum_{i\geq 2}s_{\alpha(2-i)}^\gamma\theta^\beta_{(1+i)}(X_\gamma)-\sum_{i\geq 2}s_{\alpha(3-i)}^{\talpha}\theta^\beta_{(1+i)}(X_{\talpha})=-\theta^\beta_{(3)}(X_\alpha)\equiv-\frac{1}{6}R\indices{_{\delta\alpha\gamma}^\beta}(q)x^\gamma x^\delta,
\end{equation}
in which we utilized \eqref{coorrec} to get the first relation, whereas the second one is obtained  by  \eqref{coorfun} and the last equivalence is a result of a repeated application of the relations in \eqref{recurconnA}.
In the same spirit we get
\begin{equation}\label{homsvertsec}
s_{\alpha(2)}^{\talpha}=0,\quad s_{\alpha(3)}^\beta\equiv 0.
\end{equation}
Note that the first one is obtained with the help of \eqref{coorrec} and \eqref{coorfun}, while we used \eqref{coorrec}, \eqref{coorfun}, \eqref{homsone}, the first relation in \eqref{homsvertsec} and \eqref{recurconnA} to establish the second one. Regarding  $s_{\alpha(3)}^{\talpha}$, we obtain  similarly  that
\begin{equation}\label{homshorverthree}
s_{\alpha(3)}^{\talpha}\equiv\frac{1}{12}I\indices{^{\talpha}_{\beta\gamma}}R\indices{_{\theta\alpha\delta}^\gamma}(q)x^\beta x^\delta x^\theta,
\end{equation}
where we set  $I\indices{^{\talpha}_{\beta\gamma}}:=I\indices{^{\talpha-4n}_{\beta\gamma}}$ and used \eqref{coorrec}, \eqref{coorfun} and  \eqref{recurconnA}.

We get for the term $s_{\alpha(4)}^\beta$ the following chain of relations
\begin{equation}\label{homsfourth}
s_{\alpha(4)}^\beta=-\sum_{i\geq 2}s_{\alpha(4-i)}^\gamma\theta^\beta_{(1+i)}(X_\gamma)-\sum_{i\geq 2}s_{\alpha(5-i)}^{\talpha}\theta^\beta_{(1+i)}(X_{\talpha})\equiv\frac{7}{360}R\indices{_{\delta\alpha\theta}^\gamma}(q)R\indices{_{\eta\gamma\zeta}^\beta}(q)x^\delta x^\zeta x^\eta x^\theta,
\end{equation}
where we took into account \eqref{coorrec} to obtain the first equality, while the second one is a result of \eqref{coorfun}, \eqref{homsone}, \eqref{homssecond}, \eqref{homshorverthree} and a repeated application of  \eqref{recurconnA}.

In a similar way, we establish  using \eqref{coorrec}, \eqref{coorfun}, \eqref{homsone}, \eqref{homsvertsec} and \eqref{recurconnA} that
\begin{equation}\label{homshorverfour}
s_{\alpha(4)}^{\talpha}\equiv 0.
\end{equation}
Similarly, using  \eqref{coorrec}, \eqref{coorfun},  \eqref{homsone}, \eqref{homssecond}, \eqref{homshorverthree} and \eqref{recurconnA}, we get
\begin{equation}\label{homshorverfive}
s_{\alpha(5)}^{\talpha}\equiv-\frac{1}{90}I\indices{^{\talpha}_{\theta\zeta}}R\indices{_{\delta\alpha\gamma}^\beta}(q)R\indices{_{\xi\beta\eta}^\zeta}(q)x^\gamma x^\delta x^\eta x^\theta x^\xi.
\end{equation}
To find the effect of the parabolic dilation change of the variables on the squared norm \eqref{squarednorm} we describe the result of this change on the functions $X_\alpha F^\veps$ and $X_{\talpha}F^\veps.$ We obtain with some  standard calculations  that
\begin{multline}\label{dilatchanXF}
\begin{aligned}
&\delta_\veps^*(X_\alpha F^\veps)=-4(n+1)\veps^{-2n-3}[(1+|p|^2)^2+|w|^2]^{-n-2}[(1+|p|^2)x^\alpha+\sum_{i=1}^3I\indices{^i_{\beta\alpha}}x^\beta t^i];\\
&\delta_\veps^*(X_{\talpha}F^\veps)=-4(n+1)\veps^{-2n-4}[(1+|p|^2)^2+|w|^2]^{-n-2}t^{\talpha},
\end{aligned}
\end{multline}
where in the right-hand side of the last formula we set $t^{\talpha}:=t^{\talpha-4n}.$
Now we have 
\begin{multline}\label{new11}
\delta^*_{\veps}(\xi_\alpha f^\veps)\equiv\delta^*_{\veps}(s_\alpha^\beta X_\beta F^\veps+s_\alpha^{\talpha}X_{\talpha} F^\veps)\\\equiv\delta^*_\veps[(s_{\alpha(0)}^\beta+s_{\alpha(1)}^\beta+s_{\alpha(2)}^\beta+s_{\alpha(3)}^\beta+s_{\alpha(4)}^\beta)X_\beta F^\veps+(s_{\alpha(0)}^{\talpha}+s_{\alpha(1)}^{\talpha}+s_{\alpha(2)}^{\talpha}+s_{\alpha(3)}^{\talpha}+s_{\alpha(4)}^{\talpha}+s_{\alpha(5)}^{\talpha})X_{\talpha}F^\veps]\\\equiv\delta^*_{\veps}\big(X_{\alpha}F^\veps-\frac{1}{6}R\indices{_{\delta\alpha\gamma}^\beta}(q)x^\gamma x^\delta X_\beta F^\veps+\frac{7}{360}R\indices{_{\delta\alpha\theta}^\gamma}(q)R\indices{_{\eta\gamma\zeta}^\beta}(q)x^\delta x^\zeta x^\eta x^\theta X_\beta F^\veps+\frac{1}{12}I\indices{^{\talpha}_{\beta\gamma}}R\indices{_{\theta\alpha\delta}^\gamma}(q)x^\beta x^\delta x^\theta X_{\talpha}F^\veps\\-\frac{1}{90}I\indices{^{\talpha}_{\theta\zeta}}R\indices{_{\delta\alpha\gamma}^\beta}(q)R\indices{_{\xi\beta\eta}^\zeta}(q)x^\gamma x^\delta x^\eta x^\theta x^\xi X_{\talpha}F^\veps\big)=-4(n+1)\veps^{-2n-3}[(1+|p|^2)^2+|w|^2]^{-n-2}\Big([(1+|p|^2)x^\alpha\\+\sum_{i=1}^3I\indices{^i_{\beta\alpha}}x^\beta t^i]-\frac{1}{12}\veps^2\sum_{i=1}^3I\indices{^i_{\eta\beta}}R\indices{_{\delta\alpha\gamma}^\beta}(q)x^\gamma x^\delta x^\eta t^i+\frac{1}{120}\veps^4\sum_{i=1}^3I\indices{^i_{\xi\beta}}R\indices{_{\delta\alpha\theta}^\gamma}(q)R\indices{_{\eta\gamma\zeta}^\beta}(q)x^\delta x^\zeta x^\eta x^\theta x^\xi t^i\Big).
\end{multline}
In order to get the first equivalence  in \eqref{new11} we use  $f^\veps=\psi F^\veps\equiv F^\veps$ since \eqref{numerator1}  shows that the function $\psi$ contributes only in the integral $I_2$ which is $O(\veps^5)$ according to \eqref{newI2}. The third equivalence is obtained with the help of the relations \eqref{coorfun}, \eqref{homsone}, \eqref{homssecond}, \eqref{homsvertsec}, \eqref{homshorverthree}, \eqref{homsfourth}, \eqref{homshorverfour} and \eqref{homshorverfive}. Finally, the last relation in \eqref{new11}  is a result of \eqref{dilatchanXF} and the second identity in \eqref{Curvident}. The  formula \eqref{new11} together with  \eqref{Curvident} and some standard calculations lead to
\begin{multline*}\label{deltanablaf}
\delta_\veps^*|\nabla f^\veps|^2_\eta\equiv 16(n+1)^2\veps^{-4n-6}[(1+|p|^2)^2+|w|^2]^{-2n-4}\times\\\times\big\{[(1+|p|^2)^2+|w|^2]|p|^2+\frac{17}{720}\veps^4\sum_{i=1}^3I\indices{^i_{\beta\alpha}}I\indices{^i_{\xi\theta}}R\indices{_{\delta\alpha\iota}^\gamma}(q)R\indices{_{\eta\gamma\zeta}^\theta}(q)x^\beta x^\delta x^\zeta x^\eta x^\iota x^\xi(t^i)^2\big\},
\end{multline*}
 where we omitted the terms that contain powers of $\veps$ different from $0$ and $4.$ Moreover, $t^i$ must appear only in even powers due to  the integration  reasons.
The last formula together with \eqref{volchanvar} proves the  second formula in \eqref{dilatchange} which completes the proof of the lemma.
\end{proof}

We continue with a fact from the multivariable calculus. We recall that the natural volume form on $\mathbb{R}^{4n}$ is  $dx:=dx^1\wedge\ldots\wedge dx^{4n}.$ The polar change $(x^1,\ldots,x^{4n})=(r\zeta^1,\ldots,r\zeta^{4n}),\quad\zeta:=(\zeta^1,\ldots,\zeta^{4n})\in S^{4n-1},\quad r>0,$ leads  to the representation $dx=r^{4n-1}dr\wedge d\omega,$ where $d\omega$ is a volume form on $S^{4n-1}.$  We have 
\begin{prop}\label{integrsph}
If $\alpha_1,\ldots,\alpha_{4n}$ are some non-negative integers then the next formulas hold:
\begin{equation}\label{intsphere}
\begin{aligned}
&\int_{S^{4n-1}}(\zeta^1)^{\alpha_1}\ldots(\zeta^{4n})^{\alpha_{4n}}d\omega=0,\textrm{\hspace{7.3cm}if some }\alpha_s\textrm{ is odd};\\
&\int_{S^{4n-1}}(\zeta^1)^{\alpha_1}\ldots(\zeta^{4n})^{\alpha_{4n}}d\omega=\frac{\alpha_1!\ldots\alpha_{4n}!\pi^{2n}}{2^{\alpha_1+\cdots+\alpha_{4n}-1}\big(\frac{\alpha_1}{2}\big)!\ldots\big(\frac{\alpha_{4n}}{2}\big)!\big(\frac{\alpha_1+\cdots+\alpha_{4n}+4n-2}{2}\big)!},\textrm{ if any }\alpha_s\textrm{ is even.}
\end{aligned}
\end{equation}
\end{prop}
\begin{proof}
For the Gamma function $\Gamma(t):=\int_0^\infty s^{t-1}e^{-s}ds$ one gets after changing  $s=u^2, u\in(0,\infty),$
\begin{equation}\label{Gammaf}
\Gamma (t)=2\int_0^\infty u^{2t-1}e^{-u^2}du.
\end{equation}
For the integral $I:=\int_{\mathbb{R}^{4n}}(x^1)^{\alpha_1}\ldots(x^{4n})^{\alpha_{4n}}e^{-(x^1)^2-\cdots-(x^{4n})^2}dx$ we have  as an application of \eqref{Gammaf}
\begin{equation}\label{intRn}
I=\int_{-\infty}^\infty(x^1)^{\alpha_1}e^{-(x^1)^2}dx^1\ldots\int_{-\infty}^\infty(x^{4n})^{\alpha_{4n}}e^{-(x^{4n})^2}dx^{4n}=
\begin{cases}
0,\textnormal{\hspace{1.7cm}if some }\alpha_s\textnormal{ is odd;}\\
\prod_{s=1}^{4n}\Gamma(\beta_s),\textnormal{ if any }\alpha_s\textnormal{ is even,}
\end{cases}
\end{equation}
where $\beta_s:=\frac{1}{2}(\alpha_s+1), s=1,\ldots,4n.$ On the
other hand the polar change and \eqref{Gammaf} yield
\begin{multline}\label{intRnsecond}
I=\int_{S^{4n-1}}(\zeta^1)^{\alpha_1}\ldots(\zeta^{4n})^{\alpha_{4n}}d\omega\int_0^{\infty}r^{\alpha_1+\cdots+\alpha_{4n}+4n-1}e^{-r^2}dr\\=\frac{1}{2}\Gamma(\beta_1+\cdots+\beta_{4n})\int_{S^{4n-1}}(\zeta^1)^{\alpha_1}\ldots(\zeta^{4n})^{\alpha_{4n}}d\omega.
\end{multline}
We compare \eqref{intRn} and \eqref{intRnsecond} to get the formula
\begin{equation}\label{intsphgamma}
\int_{S^{4n-1}}(\zeta^1)^{\alpha_1}\ldots(\zeta^{4n})^{\alpha_{4n}}d\omega=
\begin{cases}
0,\textnormal{\hspace{1.9cm}if some }\alpha_s\textnormal{ is odd};\\
\frac{2\Gamma(\beta_1)\ldots\Gamma(\beta_{4n})}{\Gamma(\beta_1+\cdots+\beta_{4n})},\textnormal{ if any }\alpha_s\textnormal{ is even.}
\end{cases}
\end{equation}

Now, if any $\alpha_s$ is even, the well-known formulas giving the values of the Gamma function for positive integer and half-integer arguments,
\begin{equation}\label{Gammainteger}
\Gamma(n)=(n-1)!,\quad\Gamma(n+\frac{1}{2})=\frac{(2n-1)!!}{2^n}\sqrt{\pi},
\end{equation}
imply  $\Gamma(\beta_s)=\frac{\alpha_s!\sqrt{\pi}}{(\frac{\alpha_s}{2})!2^{\alpha_s}}$ and $\Gamma(\beta_1+\cdots+\beta_{4n})=\Big(\frac{\alpha_1+\cdots+\alpha_{4n}+4n-2}{2}\Big)!
$ which inserted into \eqref{intsphgamma} give \eqref{intsphere}.
\end{proof}
The last  auxiliary result that help us to prove the main claim of this section is 
\begin{cor}\label{intcurvature}
The following formulas for integration on $S^{4n-1}$ hold:
\begin{equation}\label{intformulas}
\begin{aligned}
&\int_{S^{4n-1}}R\indices{_{\delta\alpha\gamma}^\beta}(q)R\indices{_{\eta\beta\xi}^\alpha}(q)\zeta^\gamma \zeta^\delta \zeta^\xi \zeta^\eta d\omega=\frac{3\pi^{2n}}{4(2n+1)!}||\W||^2;\\
&\int_{S^{4n-1}}I^{1\alpha\theta}I\indices{^1_{\xi\iota}}R\indices{_{\delta\theta\gamma}^\beta}(q)R\indices{_{\eta\beta\alpha}^\iota}(q)\zeta^\gamma \zeta^\delta \zeta^\xi \zeta^\eta d\omega=\frac{3\pi^{2n}}{4(2n+1)!}||\W||^2;\\
&\int_{S^{4n-1}}I^{1\alpha\theta}I\indices{^1_{\eta\iota}}R\indices{_{\delta\theta\gamma}^\beta}(q)R\indices{_{\alpha\beta\xi}^\iota}(q)\zeta^\gamma \zeta^\delta \zeta^\xi \zeta^\eta d\omega=\frac{\pi^{2n}}{2(2n+1)!}||\W||^2;\\
&\int_{S^{4n-1}}I^{1\alpha\kappa}I^{1\lambda\theta}I\indices{^1_{\gamma\beta}}I\indices{^1_{\xi\iota}}R\indices{_{\delta\theta\alpha}^\beta}(q)R\indices{_{\eta\kappa\lambda}^\iota}(q)\zeta^\gamma \zeta^\delta \zeta^\xi \zeta^\eta  d\omega=\frac{3\pi^{2n}}{4(2n+1)!}||\W||^2;\\
&\int_{S^{4n-1}}I^{1\alpha\kappa}I^{1\lambda\theta}I\indices{^1_{\delta\beta}}I\indices{^1_{\eta\iota}}R\indices{_{\alpha\theta\gamma}^\beta}(q)R\indices{_{\lambda\kappa\xi}^\iota}(q)\zeta^\gamma \zeta^\delta \zeta^\xi \zeta^\eta  d\omega=\frac{\pi^{2n}}{(2n+1)!}||\W||^2;\\
&\int_{S^{4n-1}}I^{1\alpha\kappa}I^{1\lambda\theta}I\indices{^1_{\gamma\beta}}I\indices{^1_{\eta\iota}}R\indices{_{\delta\theta\alpha}^\beta}(q)R\indices{_{\lambda\kappa\xi}^\iota}(q)\zeta^\gamma \zeta^\delta \zeta^\xi \zeta^\eta  d\omega=\frac{\pi^{2n}}{2(2n+1)!}||\W||^2;\\
&\int_{S^{4n-1}}I\indices{^i_{\beta\alpha}}I\indices{^i_{\kappa\theta}}R\indices{_{\delta\alpha\iota}^\gamma}(q)R\indices{_{\eta\gamma\xi}^\theta}(q)\zeta^\beta \zeta^\delta \zeta^\xi \zeta^\eta \zeta^\iota \zeta^\kappa d\omega=\frac{\pi^{2n}}{(2n+2)!}||\W||^2,\textrm{\hspace{0.3cm}for a fixed }i,
\end{aligned}
\end{equation}
where $||\W||^2$ denotes  the squared norm of the qc conformal curvature at $q.$
\end{cor}
\begin{proof} The curvature identities \eqref{Curvident} imply
\begin{equation}\label{conftensor}
||\W||^2=R_{\alpha\beta\gamma\delta}(q)R^{\alpha\beta\gamma\delta}(q)=2R_{\alpha\beta\gamma\delta}(q)R^{\alpha\gamma\beta\delta}(q).
\end{equation}

We  give a detail proof of the fourth formula in \eqref{intformulas}  since the proof of the others is very similar.

Denote the integral that stays in the left-hand side of the formula by $\mathcal{I}$ then we have the decomposition
\begin{equation}\label{fourthintegral}
\mathcal{I}=\mathcal{I}_1+\mathcal{I}_2+\mathcal{I}_3,
\end{equation}
where each of the integrals $\mathcal{I}_1, \mathcal{I}_2$ and $\mathcal{I}_3$ corresponds to one of the cases considered below.\footnote{Note that the intersection of these cases is not the empty set. They intersect in the case when $\gamma=\delta=\xi=\eta,$ the corresponding integral of which we divide into three equals parts and include them in the integrals $\mathcal{I}_1, \mathcal{I}_2$ and $\mathcal{I}_3$.}

\emph{Case 1:} $\gamma=\delta, \xi=\eta.$ We have
\begin{equation}\label{case1int}
\mathcal{I}_1=I^{1\alpha\kappa}I^{1\lambda\theta}I\indices{^1_{\gamma\beta}}I\indices{^1_{\xi\iota}}R\indices{_{\gamma\theta\alpha}^\beta}(q)R\indices{_{\xi\kappa\lambda}^\iota}(q)\frac{\pi^{2n}}{2(2n+1)!}=16n^2I^{1\alpha\kappa}I^{1\lambda\theta}\zeta_{1\theta\alpha}(q)\zeta_{1\kappa\lambda}(q)\frac{\pi^{2n}}{2(2n+1)!}=0,
\end{equation}
where we used \eqref{intsphere} to get the first identity and  Theorem~\ref{mainK} to obtain the third one.

\emph{Case 2:} $\gamma=\xi, \delta=\eta.$ In this case we have
\begin{multline}\label{case2int}
\mathcal{I}_2=I^{1\alpha\kappa}I^{1\lambda\theta}\underbrace{I\indices{^1_{\gamma\beta}}I\indices{^1_{\gamma\iota}}}_{\delta_{\beta\iota}}R\indices{_{\delta\theta\alpha}^\beta}(q)R\indices{_{\delta\kappa\lambda}^\iota}(q)\frac{\pi^{2n}}{2(2n+1)!}=I^{1\alpha\kappa}I^{1\lambda\theta}R\indices{_{\delta\theta\alpha}^\beta}(q)R\indices{_{\delta\kappa\lambda}^\beta}(q)\frac{\pi^{2n}}{2(2n+1)!}\\=\underbrace{I^{1\alpha\beta}I^{1\lambda\beta}}_{\delta_{\alpha\lambda}}R\indices{_{\delta\theta\alpha}^\kappa}(q)R\indices{_{\delta\kappa\lambda}^\theta}(q)\frac{\pi^{2n}}{2(2n+1)!}=R_{\delta\theta\kappa\alpha}(q)R^{\delta\kappa\theta\alpha}(q)\frac{\pi^{2n}}{2(2n+1)!}=\frac{\pi^{2n}}{4(2n+1)!}||\W||^2,
\end{multline}
where we used \eqref{intsphere} to get the first identity and \eqref{Curvident} to obtain the third and the fourth one. The fifth identity  follows from the second
relation in \eqref{conftensor}.

\emph{Case 3:} $\gamma=\eta, \delta=\xi.$ We establish in this situation
\begin{multline}\label{case3int}
\mathcal{I}_3=I^{1\alpha\kappa}I^{1\lambda\theta}I\indices{^1_{\gamma\beta}}I\indices{^1_{\delta\iota}}R\indices{_{\delta\theta\alpha}^\beta}(q)R\indices{_{\gamma\kappa\lambda}^\iota}(q)\frac{\pi^{2n}}{2(2n+1)!}=\underbrace{I^{1\alpha\beta}I\indices{^1_{\gamma\beta}}}_{\delta_{\alpha\gamma}}\underbrace{I^{1\lambda\iota}I\indices{^1_{\delta\iota}}}_{\delta_{\lambda\delta}}R\indices{_{\delta\theta\alpha}^\kappa}(q)R\indices{_{\gamma\kappa\lambda}^\theta}(q)\frac{\pi^{2n}}{2(2n+1)!}\\=R\indices{_{\delta\theta\alpha}^\kappa}(q)R\indices{_{\alpha\kappa\delta}^\theta}(q)\frac{\pi^{2n}}{2(2n+1)!}=R_{\alpha\kappa\delta\theta}(q)R^{\alpha\kappa\delta\theta}(q)\frac{\pi^{2n}}{2(2n+1)!}=\frac{\pi^{2n}}{2(2n+1)!}||\W||^2,
\end{multline}
where we used \eqref{intsphere}, the fifth and the fourth equalities in \eqref{Curvident} and the first identity in \eqref{conftensor} to get the first, the second, the fourth and the fifth identity, respectively.

Now, we substitute \eqref{case1int}, \eqref{case2int} and \eqref{case3int} in \eqref{fourthintegral} to obtain the desired formula.

Note that in order to obtain the seventh formula in \eqref{intformulas} we consider fifteen cases in a similar way.
\end{proof}
Now we are ready to prove the main result of this section:

\begin{thrm}\label{expansion}Let $(M,\eta)$ be a quaternionic contact manifold of dimension $4n+3$ with a qc contact form $\eta$
normalized at a fixed point $q\in M$ according to
Theorem~\ref{mainK}.  Then the qc Yamabe functional
\eqref{Yamabefunc1} over the test functions $f^\veps$ defined near
$q$ in Section~\ref{Asymptexp} has the asymptotic expansion
\eqref{expansionfunc1}.
\end{thrm}
\begin{proof}
We begin with the remark that the natural volume form $Vol_\Theta$ on the Heisenberg group $\mathbf{G}(\mathbb{H})$ can be expressed in the terms of the qc parabolic normal coordinates as follows
\begin{equation}\label{volheis}
Vol_{\Theta}=\frac{(2n)!}{8}dt^1\wedge dt^2\wedge dt^3\wedge dx^1\wedge\ldots\wedge dx^{4n}.
\end{equation}
Our first aim is to investigate the numerator of \eqref{Yamabefunc2A}. We get consecutively

\begin{multline}\label{exactnum}
\int_M\Big[4\frac{n+2}{n+1}|\nabla f^{\varepsilon}|_{\eta}^2+S(f^{\varepsilon})^2\Big]\,Vol_{\eta}=4\frac{n+2}{n+1}\int_M\delta^*_\veps(|\nabla f^\veps|^2_\eta\vol)+O(\veps^5)\\\equiv64(n+1)(n+2)\int_{\mathbf{G}(\mathbb{H})}[(1+|p|^2)^2+|w|^2]^{-2n-4}\big\{[(1+|p|^2)^2+|w|^2]|p|^2\\+\frac{17}{720}\veps^4\underbrace{\sum_{i=1}^3I\indices{^i_{\beta\alpha}}I\indices{^i_{\xi\theta}}R\indices{_{\delta\alpha\iota}^\gamma}(q)R\indices{_{\eta\gamma\zeta}^\theta}(q)x^\beta x^\delta x^\iota x^\zeta x^\eta x^\xi(t^i)^2}_{:=\tilde{\chi}(x,t)}\big\}[1+\veps^4\chi(x)]\,Vol_{\Theta}\\\equiv8(n+1)(n+2)(2n)!\int_{\mathbb{R}^{4n}}\int_{\mathbb{R}^3}[(1+|p|^2)^2+|w|^2]^{-2n-4}\Big\{[(1+|p|^2)^2+|w|^2]|p|^2\\+\big\{[(1+|p|^2)^2+|w|^2]|p|^2\chi(x)+\frac{17}{720}\tilde{\chi}(x,t)\big\}\veps^4\Big\}dt^1\wedge dt^2\wedge dt^3\wedge dx^1\wedge\ldots\wedge dx^{4n}\\=8(n+1)(n+2)(2n)!\int_{S^{4n-1}}\int_0^\infty\int_{\mathbb{R}^3}[(1+r^2)^2+(t^1)^2+(t^2)^2+(t^3)^2]^{-2n-3}r^{4n+1}dt^1\wedge dt^2\wedge dt^3\wedge dr\wedge d\omega\\+\veps^48(n+1)(n+2)(2n)!\int_{S^{4n-1}}\int_0^\infty\int_{\mathbb{R}^3}[(1+r^2)^2+(t^1)^2+(t^2)^2+(t^3)^2]^{-2n-3}r^{4n+5}\chi(\zeta)dt^1\wedge dt^2\wedge dt^3\wedge dr\wedge d\omega\\+\veps^4\frac{17}{90}(n+1)(n+2)(2n)!\int_{S^{4n-1}}\int_0^\infty\int_{\mathbb{R}^3}[(1+r^2)^2+(t^1)^2+(t^2)^2+(t^3)^2]^{-2n-4}r^{4n+5}\tilde{\chi}(\zeta,t)dt^1\wedge dt^2\wedge dt^3\wedge dr\wedge d\omega\\=:\mathcal{S}_1+(\mathcal{S}_2+\mathcal{S}_3)\veps^4,
\end{multline}
where we used the parabolic dilation change of the variables and
\eqref{numerator2Last} to get the first relation and the second
equivalence in \eqref{dilatchange} to obtain the second one. The
third equivalence  is a result of \eqref{volheis}, whereas the
fourth one is established after the polar change in the $x$
variable. Now we are going to calculate explicitly the integrals
$\mathcal{S}_1, \mathcal{S}_2$ and $\mathcal{S}_3$ that appear in
the formula above. We recall firstly the formula \cite[(5.4)]{JL}
\begin{equation}\label{defintgamma}
\int_0^\infty b^\gamma(a^2+b^2)^{-\alpha/2}db=\frac{\Gamma((\gamma+1)/2)\Gamma((\alpha-\gamma-1)/2)}{2\Gamma(\alpha/2)}a^{\gamma-\alpha+1},
\end{equation}
where $a, b\in\mathbb{R}^+$ and $\alpha, \gamma\in \mathbb{R},\quad\alpha-\gamma-1>0,\quad\gamma+1>0.$

We obtain for $\mathcal{S}_1:$
\begin{multline}\label{explsone}
\mathcal{S}_1=32n(n+1)(n+2)\pi^{2n}\int_0^\infty\int_{\mathbb{R}^3}[(1+r^2)^2+(t^1)^2+(t^2)^2+(t^3)^2]^{-2n-3}r^{4n+1}dt^1dt^2dt^3dr\\=\frac{2n(n+2)\pi^{2n+2}}{4^{2n}(2n+1)},
\end{multline}
where we applied \eqref{intsphere} to get the first equality,
while the second one is obtained by \eqref{Gammainteger} and a
fourfold application of \eqref{defintgamma}. More concretely, we
integrate at the first stage over the variable $t^1$ and set
$a:=[(1+r^2)^2+(t^2)^2+(t^3)^2]^\frac12,\quad b:=t^1$ which reduce
the four-fold integral to a three-fold one and s.o.

For $\mathcal{S}_2,$ 
we get by  \eqref{defchi}  and the first six equalities in \eqref{intformulas} that
\begin{equation}\label{new12}\int_{S^{4n-1}}\chi(\zeta)d\omega=\frac{-11\pi^{2n}}{1440(2n+1)!}||\W||^2.
\end{equation}
We obtain with the help of \eqref{new12} in a similar way as \eqref{explsone} the following representation of $\mathcal{S}_2:$
\begin{equation}\label{explstwo}
\mathcal{S}_2=\frac{-11(n+1)(n+2)\pi^{2n+2}}{45.4^{2n+3}n(2n+1)^2}||\W||^2.
\end{equation}

We use the seventh formula in \eqref{intformulas} to get after certain calculations similarly to   \eqref{explsone} and \eqref{explstwo} that
\begin{equation}\label{explsthree}
\mathcal{S}_3=\frac{17(n+2)\pi^{2n+2}}{30.4^{2n+3}n(2n+1)^2(2n+3)}||\W||^2.
\end{equation}

Now we substitute \eqref{explsone}, \eqref{explstwo} and \eqref{explsthree} into \eqref{exactnum} to obtain
\begin{multline}\label{exactnumA}
\int_M\Big[4\frac{n+2}{n+1}|\nabla f^{\varepsilon}|_{\eta}^2+S(f^{\varepsilon})^2\Big]\,Vol_{\eta}\\=\frac{2n(n+2)\pi^{2n+2}}{4^{2n}(2n+1)}+\frac{(n+2)(-44n^2-110n-15)\pi^{2n+2}}{90.4^{2n+3}n(2n+1)^2(2n+3)}||\W||^2\veps^4+O(\veps^5).
\end{multline}

Our next goal is to calculate the integral that stays in the denominator of \eqref{Yamabefunc2A}. We have
\begin{multline}\label{exactdenom}
\int_M(f^\veps)^{2^*}\vol=\int_M\delta_\veps^*[(f^{\veps})^{2^*}\vol]\equiv\int_{\mathbf{G}(\mathbb{H})}F^{2^*}[1+\veps^4\chi(x)]\,Vol_{\Theta}\\=\frac{(2n)!}{8}\int_{\mathbb{R}^{4n}}\int_{\mathbb{R}^3}[(1+|p|^2)^2+|w|^2]^{-2n-3}[1+\veps^4\chi(x)]dt^1\wedge dt^2\wedge dt^3\wedge dx^1\wedge\ldots\wedge dx^{4n}\\=\frac{(2n)!}{8}\int_{S^{4n-1}}\int_0^\infty\int_{\mathbb{R}^3}[(1+r^2)^2+(t^1)^2+(t^2)^2+(t^3)^2]^{-2n-3}r^{4n-1}dt^1\wedge dt^2\wedge dt^3\wedge dr\wedge d\omega\\+\veps^4\frac{(2n)!}{8}\int_{S^{4n-1}}\int_0^\infty\int_{\mathbb{R}^3}[(1+r^2)^2+(t^1)^2+(t^2)^2+(t^3)^2]^{-2n-3}r^{4n+3}\chi(\zeta)dt^1\wedge dt^2\wedge dt^3\wedge dr\wedge d\omega\\=:\tilde{\mathcal{S}}_1+\tilde{\mathcal{S}}_2\veps^4,
\end{multline}
where we made change of the variables to get the first relation
and applied the first equivalence from \eqref{dilatchange} to
obtain the second one. Furthermore, we used \eqref{volheis} to
establish the third identity in \eqref{exactdenom}, whereas the
fourth one is an effect of the polar change in the $x$ variable.

The calculations of  $\tilde{\mathcal{S}}_1$ and $\tilde{\mathcal{S}}_2$ are analogous to those we made in obtaining \eqref{explsone}, \eqref{explstwo} and \eqref{explsthree}. We get as a result the formulas
\begin{equation*}
\tilde{\mathcal{S}}_1=\frac{\pi^{2n+2}}{2.4^{2n+2}(2n+1)}\qquad\textnormal{and}\qquad\tilde{\mathcal{S}}_2=\frac{-11\pi^{2n+2}}{90.4^{2n+5}(2n+1)^2(2n+2)}||\W||^2
\end{equation*}
which we substitute in \eqref{exactdenom} to establish the identity
\begin{equation}\label{exactdenomA}
\int_M(f^\veps)^{2^*}\vol=\frac{\pi^{2n+2}}{2.4^{2n+2}(2n+1)}+\frac{-11\pi^{2n+2}}{90.4^{2n+5}(2n+1)^2(2n+2)}||\W||^2\veps^4+O(\veps^5).
\end{equation}

Now we insert \eqref{exactnumA} and \eqref{exactdenomA} into \eqref{Yamabefunc2A} to get
\begin{multline*}
\Upsilon_\eta(f^\veps)=\Big[\frac{2n(n+2)\pi^{2n+2}}{4^{2n}(2n+1)}+\frac{(n+2)(-44n^2-110n-15)\pi^{2n+2}}{90.4^{2n+3}n(2n+1)^2(2n+3)}||\W||^2\veps^4+O(\veps^5)\Big]\times\\\times\Big[\frac{\pi^{2n+2}}{2.4^{2n+2}(2n+1)}+\frac{-11\pi^{2n+2}}{90.4^{2n+5}(2n+1)^2(2n+2)}||\W||^2\veps^4+O(\veps^5)\Big]^{-2/2^*}\\=\frac{\pi^{-\frac{(2n+2)^2}{2n+3}}}{[2.4^{2n+2}(2n+1)]^{-\frac{2n+2}{2n+3}}}\Big[\frac{2n(n+2)\pi^{2n+2}}{4^{2n}(2n+1)}+\frac{(n+2)(-44n^2-110n-15)\pi^{2n+2}}{90.4^{2n+3}n(2n+1)^2(2n+3)}||\W||^2\veps^4+O(\veps^5)\Big]\times\\\times\Big[1+\frac{11||\W||^2\veps^4}{2880(2n+1)(2n+3)}+O(\veps^5)\Big]=\Lambda(1-c(n)||\W||^2\veps^4)+O(\veps^5),
\end{multline*}
where $c(n):=\frac{22n+3}{2304n^2(2n+1)(2n+3)}.$ The theorem is proved.
\end{proof}


\begin{thebibliography}{99}



\bibitem{Au} Aubin, Th.,  \emph{\'{E}quations diff\'{e}rentielles non lin\'eaires et probl\`eme de Yamabe concernant la courbure scalaire},
 J. Math. Pures Appl. (9) {\bf 55} (1976), no.~3, 269--296.

\bibitem{Bah} Bahri, A., \emph{Proof of the Yamabe conjecture for locally
conformally flat manifolds}, Non. Analysis, Theory, Methods and
Appli., 20 (10) (1993), 1261-1278.

\bibitem{BahBre}Bahri, A., Brezis, H., \emph{Nonlinear elliptic equations}, in �Topics in Geometry in
memory of Joseph D�Atri�, Simon Gindikin editor, Birkh�auser,
Boston-Basel-Berlin, (1996), 1-100.


\bibitem{Biq1} Biquard, O., \emph{M\'{e}triques d'Einstein asymptotiquement
sym\'{e}triques}, Ast\'{e}risque \textbf{265} (2000).

\bibitem{Biq2} Biquard, O., \emph{Quaternionic contact structures},
Quaternionic structures in mathematics and physics (Rome, 1999),
23--30 (electronic), Univ. Studi Roma "La Sapienza", Roma, 1999.




\bibitem{ChM} Chern, S.S. \& Moser, J., \emph{Real hypersurfaces in complex manifolds},
Acta Math. \textbf{133} (1974), 219--271.





\bibitem{D} Duchemin, D., \emph{Quaternionic contact structures in dimension
7}, Ann. Inst. Fourier (Grenoble) \textbf{56} (2006), no. 4,
851--885.

\bibitem{FGr85}
Fefferman, C., \& Graham, C.R.,  \emph{Conformal invariants.} The
mathematical heritage of \'Elie Cartan (Lyon, 1984). Ast\'erisque
1985, Num\'ero Hors S\'erie, 95--116.

\bibitem{F2}
Folland, G.B,, \emph{Subelliptic estimates and function spaces on
nilpotent Lie groups},
  Ark. Math., \textbf{13}~(1975), 161--207.

\bibitem{FS}
Folland, G.B., \& Stein, E.M., \emph{Estimates for the $\Bar
{\partial}_{b}$ Complex and Analysis on the Heisenberg Group},
Comm. Pure Appl. Math., \textbf{27}~(1974), 429--522.



\bibitem{Ga}
Gamara,  N., \emph{The CR Yamabe conjecture the case $n=1$}, J.
Eur. Math. Soc. (JEMS) {\bf 3} (2001), no.~2, 105--137.

\bibitem{GaY}
Gamara, N.\ \&\  Yacoub, R., \emph{CR Yamabe conjecture -- the
conformally flat case}, Pacific J. Math. {\bf 201} (2001), no.~1,
121--175.

\bibitem{GV}
Garofalo, N. \& Vassilev, D.,\ \emph{ Symmetry properties of
positive entire solutions of Yamabe type equations on groups of
Heisenberg type}, Duke Math J, {\bf 106} (2001), no. 3, 411--449.

\bibitem{GrL91}
 Graham, C. R., \&  Lee, John M., \emph{Einstein metrics with prescribed conformal infinity on the ball.} Adv. Math. 87 (1991), no. 2, 186--225.



\bibitem{IMV} Ivanov, S., Minchev, I., \& Vassilev, D.,\ \emph{Quaternionic
contact Einstein structures and the quaternionic contact Yamabe
problem},  Memoirs Am. Math. Soc., vol. 231, number 1086 (2014).



\bibitem{IMV1} Ivanov, S., Minchev, I., \& Vassilev, D.,\  \emph{Extremals for the Sobolev inequality on the
seven dimensional quaternionic Heisenberg group and the
quaternionic contact Yamabe problem}, J. Eur. Math. Soc. (JEMS) 12
(2010), no. 4, 1041--1067.

\bibitem{IMV2} Ivanov, S., Minchev, I., \& Vassilev, D.,\  \emph{The optimal constant in the $L^2$
Folland-Stein inequality on the quaternionic Heisenberg group},
Ann. Sc. Norm. Super. Pisa Cl. Sci. (5) Vol. XI (2012), 635--662.

\bibitem{IMV4} Ivanov, S., Minchev, I., \& Vassilev, D.,\
\emph{The qc Yamabe problem on 3-Sasakian manifolds and the
quaternionic Heisenberg group}, arXiv:1504.03142.




\bibitem{IV} Ivanov, S., \& Vassilev, D., \emph{Conformal quaternionic
contact curvature and the local sphere theorem}, J. Math. Pures
Appl. \textbf{93} (2010), 277--307.



\bibitem{IV2} Ivanov, S., \& Vassilev, D., \emph{Extremals for the Sobolev Inequality and the
Quaternionic Contact Yamabe Problem}, Imperial College Press
Lecture Notes, World Scientific Publishing Co. Pte. Ltd.,
Hackensack, NJ, 2011.

\bibitem{IV5}  Ivanov, S., \& Vassilev, D., \emph{The Lichnerowicz and Obata first eigenvalue theorems and the
Obata uniqueness result in the Yamabe problem on CR and
quaternionic contact manifolds},  Nonlinear Analysis \textbf{126} (2015) 262-323.

\bibitem{JL} Jerison, D., Lee, J. M., \emph{Intrinsic CR normal coordinates and the CR Yamabe problem}, J. Diff. Geom. \textbf{29} (1989), 303-343.


\bibitem{Kunk} Kunkel, Chr., \emph{Quaternionic contact normal coordinates}, preprint, arXiv:0807.0465 [math.DG].

\bibitem{LeB91}
LeBrun, C., \emph{On complete quaternionic-K\"ahler manifolds.}
Duke Math. J. 63 (1991), no. 3, 723--743.



\bibitem{LP} Lee, J.M., Parker, T., \emph{The Yamabe problem}, Bull. Amer. Math. Soc. \textbf{17} (1987), 37-91.



\bibitem{M} Mostow, G. D., \emph{Strong rigidity of locally symmetric spaces%
}, Annals of Mathematics Studies, No. 78. Princeton University
Press, Princeton, N.J.; University of Tokyo Press, Tokyo, 1973.
v+195 pp.


\bibitem{P} Pansu, P., \emph{M\'etriques de Carnot-Carath\'eodory et
quasiisom\'etries des espaces sym\'etriques de rang un}, Ann. of
Math. (2) 129 (1989), no. 1, 1--60.



\bibitem{Sh}
Schoen,  R., \emph{Conformal deformation of a Riemannian metric to
constant scalar curvature},
 J. Differential Geom., \textbf{20}~ (1984), no. 2, 479--495.

\bibitem{T}
Tanaka, N., \emph{A differential geometric study on strongly
pseudo-convex manifolds}, Lectures in Mathematics, Department of
Mathematics, Kyoto University, No. 9. Kinokuniya Book-Store Co.,
Ltd., Tokyo, 1975


\bibitem{Wei} Wang, W., \emph{The Yamabe problem on quaternionic contact
manifolds}, Ann. Mat. Pura Appl., \textbf{186} (2007), no. 2,
359--380.

\bibitem{W} Webster, S. M., \emph{Pseudo-hermitian structures on a real
hypersurface}, J.Diff. Geom., {\bf 13} (1979), 25--41.


\end{thebibliography}
\end{document}